\documentclass[12pt,a4paper]{amsart}

\title{Quasi-hyperbolic planes in relatively hyperbolic groups}
\author{John M. Mackay}
\address{School of Mathematics \\ University of Bristol \\ Bristol, UK}
\email{john.mackay@bristol.ac.uk}
\author{Alessandro Sisto}
\address{Department of Mathematics, ETH Zurich, 8092 Zurich, Switzerland}
\email{sisto@math.ethz.ch}
\date{\today}
\thanks{The research of the first author was supported in part by EPSRC grants EP/K032208/1 and EP/P010245/1.}

\usepackage{amsmath,amsthm,amsfonts,graphicx,amssymb,amscd}
\usepackage{hyperref}

\numberwithin{equation}{section}
\newtheorem{theorem}[equation]{Theorem}
\newtheorem{proposition}[equation]{Proposition}
\newtheorem{corollary}[equation]{Corollary}
\newtheorem{lemma}[equation]{Lemma}

\newtheorem{example}[equation]{Example}

\newtheorem{definition}[equation]{Definition}
\newtheorem{remark}[equation]{Remark}

\newtheoremstyle{citing}
  {3pt}
  {3pt}
  {\itshape}
  {}
  {\bfseries}
  {}
  {.5em}
  {\thmnote{#3}}

\theoremstyle{citing}
\newtheorem*{varthm}{}

\DeclareMathOperator{\diam}{diam}

\DeclareMathOperator{\arccosh}{arccosh}
\DeclareMathOperator{\Isom}{Isom}

\newcommand{\del}{\delta}
\newcommand{\gam}{\gamma}
\newcommand{\eps}{\epsilon}

\newcommand{\lam}{\lambda}

\newcommand{\bdry}{\partial_\infty}

\newcommand{\cH}{\mathcal{H}}

\newcommand{\cB}{\mathcal{B}}
\newcommand{\cC}{\mathcal{C}}
\newcommand{\cV}{\mathcal{V}}
\newcommand{\cD}{\mathcal{D}}
\newcommand{\cO}{\mathcal{O}}

\newcommand{\cP}{\mathcal{P}}

\newcommand{\ra}{\rightarrow}
\newcommand{\R}{\mathbb{R}}
\newcommand{\Sph}{\mathbb{S}}
\newcommand{\N}{\mathbb{N}}
\newcommand{\Z}{\mathbb{Z}}
\newcommand{\HH}{\mathbb{H}}

\newcommand{\tilM}{\widetilde{M}}

\def\XXint#1#2#3{{\setbox0=\hbox{$#1{#2#3}{\int}$}
\vcenter{\hbox{$#2#3$}}\kern-.5\wd0}}

\numberwithin{equation}{section}

\begin{document}

\begin{abstract}
We show that any group that is hyperbolic relative to virtually nilpotent subgroups,
and does not admit peripheral splittings,
contains a quasi-isometrically embedded copy of the hyperbolic plane.  In natural situations, the specific embeddings 
we find remain quasi-isometric embeddings when composed with the inclusion map from the Cayley 
graph to the coned-off graph, as well as when composed with the quotient map to ``almost every'' peripheral (Dehn) filling.

We apply our theorem to study the same question for fundamental groups
of $3$-manifolds.

The key idea is to study quantitative geometric properties of the boundaries of
relatively hyperbolic groups, such as linear connectedness.
In particular, we prove a new existence result for quasi-arcs that avoid obstacles.
\end{abstract}

\maketitle
\tableofcontents

\section{Introduction}\label{sec-intro}

A well known question of Gromov asks whether every (Gromov) hyperbolic group which is not
virtually free contains a surface group.
While this question is still open, its geometric analogue has a complete solution.
Bonk and Kleiner~\cite{BK-05-quasiarc-planes}, answering
a question of Papasoglu, showed the following.

\begin{theorem}[{Bonk--Kleiner \cite{BK-05-quasiarc-planes}}]\label{thm-bonk-kleiner}
	A hyperbolic group $G$ contains a quasi-isometrically embedded copy of $\HH^2$ 
	if and only if it is not virtually free.
\end{theorem}

In this paper, we study when a relatively hyperbolic group admits a quasi-isometrically
embedded copy of $\HH^2$ by analysing the geometric properties of boundaries of such groups.

For a hyperbolic group $G$, quasi-isometrically embedded copies of $\HH^2$ in $G$ correspond
to quasisymmetrically embedded copies of the circle $S^1 = \bdry \HH^2$ in the boundary of the group.
Bonk and Kleiner build such a circle when a hyperbolic group has connected boundary
by observing that the boundary is \emph{doubling} (there exists $N$ so that any ball can be covered by
$N$ balls of half the radius) and \emph{linearly connected} (there exists $L$ so that any points $x$ and $y$
can be joined by a continuum of diameter at most $Ld(x,y)$).  For such spaces, a theorem of Tukia
applies to find quasisymmetrically embedded arcs, or \emph{quasi-arcs} \cite{Tuk-96-qarc}.

We note that this proof relies on the local connectedness of boundaries of one-ended hyperbolic groups,
a deep result following from work of Bestvina and Mess, and Bowditch and Swarup 
\cite[Proposition 3.3]{BM-91-dimbdry}, \cite[Theorem 9.3]{Bow-98-bdry-access-hyp}, \cite[Corollary 0.3]{Bow-99-conn-lim-set}, 
\cite{Swa-96-cut-point}.

Our strategy is similar to that of Bonk and Kleiner, but to implement this we have to prove several basic results 
regarding the geometry of the boundary of a relatively hyperbolic group,
which we believe are of independent interest.

The model for the boundary that we use is due to Bowditch, who builds a model space $X(G, \cP)$ by gluing
horoballs into a Cayley graph for $G$, and setting $\bdry(G, \cP) = \bdry X(G, \cP)$
\cite{Bow-99-rel-hyp} (see also \cite{Gro-Man-08-dehn-rel-hyp}).

We fix a choice of $X(G, \cP)$ and, for suitable conditions on the peripheral subgroups, 
we show that the boundary $\bdry(G, \cP)$ has good geometric properties.  
For example, using work of Dahmani and Yaman, such boundaries will be doubling if and only if the peripheral subgroups
are virtually nilpotent.
We establish linear connectedness when the peripheral subgroups are one-ended and there are no peripheral splittings.
(See Sections~\ref{sec-boundaries} and \ref{sec-lin-conn} for precise statements.)

At this point, by Tukia's theorem, we can find quasi-isometrically embedded copies of $\HH^2$ in $X(G, \cP)$,
but these can stray far away from $G$ into horoballs in $X(G, \cP)$.  To find copies of $\HH^2$ actually in $G$ itself we must prevent
this by building a quasi-arc in the boundary that in a suitable sense stays relatively far away from the
parabolic points.  

This requires additional geometric properties of the boundary (see Section~\ref{sec-avoiding}),
and also a generalisation of Tukia's theorem which builds quasi-arcs that avoid certain kinds of obstacles
(Theorem~\ref{thm-modified-qarc}).

A simplified version of our main result is the following:

\begin{theorem}\label{thm-main-hyp-rel-hyp-simple}
	Let $(G, \cP)$ be a finitely generated relatively hyperbolic group,
	where all $P\in\cP$ are virtually nilpotent.
	Suppose $G$ is one-ended and does not split over a subgroup of a conjugate of some $P\in\cP$. 
	
	Then there is a quasi-isometric embedding of $\HH^2$ in $G$.
\end{theorem}

The methods we develop are able to find embeddings that avoid more subgroups than just virtually nilpotent peripheral groups.  Here is a more precise version of the theorem: 

\begin{theorem}\label{thm-main-hyp-rel-hyp}
	Suppose both $(G, \cP_1)$ and $(G, \cP_1 \cup \cP_2)$ are finitely generated relatively hyperbolic groups,
	where all peripheral subgroups in $\cP_1$ are virtually nilpotent and non-elementary,
	and all peripheral subgroups in $\cP_2$ are hyperbolic.
	Suppose $G$ is one-ended and does not split over a subgroup of a conjugate of some $P\in\cP_1$. 
	
	Finally, suppose that $\bdry H \subset \bdry(G, \cP_1)$ does not
	locally disconnect the boundary, for any $H \in \cP_2$ (see Definition~\ref{def-loc-disconn}).
	Then there is a quasi-isometric embedding of $\HH^2$ in $G$ that is {transversal} in $(G, \cP_1 \cup \cP_2)$.
\end{theorem}

(Theorem \ref{thm-main-hyp-rel-hyp-simple} follows from Theorem \ref{thm-main-hyp-rel-hyp} by letting $\cP_2=\emptyset$ and $\cP_1$ equal to the collection of non-elementary elements of $\cP$.)

Roughly speaking, a quasi-isometric embedding is \emph{transversal} if the image has only bounded intersection with any (neighbourhood of a) left coset of a peripheral subgroup 
(see Definition \ref{transversal}).
If both $\cP_1$ and $\cP_2$ are empty the group is hyperbolic and the result is a corollary of Theorem~\ref{thm-bonk-kleiner}.
If $\cP_1$ is empty, but $\cP_2$ is not, then the group is hyperbolic, but the quasi-isometric embeddings we find
avoid the hyperbolic subgroups conjugate to those in $\cP_2$.
\begin{example}
	Let $M$ be a compact hyperbolic $3$-manifold with a single, totally geodesic surface as boundary $\partial M$.
	The fundamental group $G = \pi_1(M)$ is hyperbolic, and also is hyperbolic relative to $H = \pi_1(\partial M)$
	(see, for example, \cite[Proposition 13.1]{Bel-RH}). 
	
	The hypotheses of Theorem \ref{thm-main-hyp-rel-hyp} are satisfied for $\cP_1 = \emptyset$ and $\cP_2 = \{H\}$,
	since $\bdry G = \bdry(G, \emptyset)$ is a Sierpi\'nski carpet, 
	with the boundary of conjugates of $H$ corresponding to the peripheral circles of the carpet.
	Thus, we find a transversal quasi-isometric embedding of $\HH^2$ into $G$.
\end{example}

The notion of transversality is 
interesting for us, because a transverse quasi-isometric embedding
of a geodesic metric space $Z \ra G$ induces:
\begin{enumerate}
	\item a quasi-isometric embedding $Z \ra G \ra \hat{\Gamma}$ into the coned-off
	(or ``electrified'') graph $\hat{\Gamma}$ (see Proposition~\ref{coned-off}), and
	\item a quasi-isometric embedding $Z \ra G \ra G/\!\ll\! \{N_i\}\!\gg$ into
	certain peripheral (or Dehn) fillings of $G$ (see Proposition~\ref{perfill}).
\end{enumerate}

When combined, Theorem \ref{thm-main-hyp-rel-hyp} and Proposition \ref{perfill} provide 
interesting examples of relatively hyperbolic groups containing quasi-iso\-met\-ric\-al\-ly embedded copies of $\HH^2$
that do not have virtually nilpotent peripheral subgroups.
A key point here is that Theorem \ref{thm-main-hyp-rel-hyp} provides embeddings transversal to hyperbolic subgroups,
and so one can find many interesting peripheral fillings.
See Example \ref{nonvnilpquot} for details.

Using our results, we describe when the fundamental group of
a closed, oriented $3$-manifold contains a quasi-iso\-met\-rical\-ly embedded copy of $\HH^2$.
Determining which $3$-manifolds (virtually) contain immersed or embedded $\pi_1$-injective surfaces
is a very difficult problem \cite{KM, CPL, Lac, BS-graph,Coo-Fut-17-quasi-fuchs-cusp}.
The following theorem essentially follows from known results, in particular
work of Masters and Zhang~\cite{Ma-Zh-08-fuch-hyp-knot,Ma-Zh-09-fuch-hyp-link}.
However, our proof is a simple consequence of Theorem~\ref{thm-main-hyp-rel-hyp} and the geometrisation theorem.

\begin{varthm}[Theorem \ref{thm-threemanifolds}.]
	Let $M$ be a closed $3$-manifold. Then $\pi_1(M)$ contains a quasi-isometrically 
	embedded copy of $\HH^2$ if and only if $M$ does not split as the connected sum
of manifolds each with geometry $S^3, \R^3, S^2\times \R$ or $\mathrm{Nil}$.
\end{varthm}

Notice that the geometries mentioned above are exactly those that give virtually nilpotent fundamental groups.

We note that recently Leininger and Schleimer proved a result similar to Theorem~\ref{thm-main-hyp-rel-hyp} for Teichm\"{u}ller spaces \cite{LS-teich}, using very different techniques.

In an earlier version of this paper we claimed a characterisation of which groups hyperbolic relative to virtually nilpotent subgroups admitted a quasi-isometrically embedded copy of $\HH^2$.  This claim was incorrect due to problems with amalgamations over elementary subgroups.
One cause of trouble is the following:
\begin{example}
	Let $F_2 = \langle a,b \rangle$ be the free group on two generators, and let $H$ be the Heisenberg group, with centre $Z(H) \cong \Z$.  Let $G = F_2 *_{\Z} H$, where we amalgamate the distorted subgroup $Z(H)$ and $\langle [a,b] \rangle \leq F_2$.  The group $G$ is hyperbolic relative to $\{H\}$, is one-ended, but does not contain any quasi-isometrically embedded copy of $\HH^2$.
\end{example}
It remains an open question to characterise which relatively hyperbolic groups with virtually nilpotent peripheral groups admit quasi-isometrically embedded copies of $\HH^2$.

Finally, we note that the geometric properties of boundaries of relatively hyperbolic groups we establish here have recently been used by Groves--Manning--Sisto to study the relative Cannon conjecture for relatively hyperbolic groups~\cite{Gro-Man-Sis-16-rel-hyp-cannon}.
\subsection{Outline}

In Section~\ref{sec-rel-hyp-defs} we define relatively hyperbolic groups and their boundaries,
and discuss transversality and its consequences.
In Section~\ref{sec-separation-in-bdry} we give preliminary results linking 
the geometry of the boundary of a relatively hyperbolic group
to that of its model space.
Further results on the boundary itself are found in Sections~\ref{sec-boundaries}--\ref{sec-avoiding}, 
in particular, how sets can be connected, and avoided, in a controlled manner.

The existence of quasi-arcs that avoid obstacles is proved in Section~\ref{sec-qarcs-that-avoid}.
The proof of Theorem~\ref{thm-main-hyp-rel-hyp} is given in Section~\ref{sec-build-planes}.
Finally, connections with $3$-manifold groups are explored in Section~\ref{sec-three-manifolds}.

\subsection{Notation}

The notation $x\gtrsim_C y$ (occasionally abbreviated to $x \gtrsim y$) signifies $x\geq y-C$.
Similarly, $x\lesssim_C y$ signifies $x\leq y+C$.
If $x\lesssim_C y$ and $x\gtrsim_C y$ we write $x\approx_C y$.

Throughout, $C$, $C_1$, $C_2$, etc., will refer to appropriately chosen constants.
The notation $C_3=C_3(C_1, C_2)$ indicates that $C_3$ depends on the choices of $C_1$ and $C_2$.

For a metric space $(Z,d)$, the \emph{open ball} with centre $z\in Z$ and radius $r>0$ is denoted
by $B(z,r)$.  The \emph{closed ball} with the same centre and radius is denoted
by $\overline{B}(z,r)$.
We write $d(z, V)$ for the infimal distance between a subset $V \subset Z$ and a point $z \in Z$.
The \emph{open neighbourhood} of $V \subset Z$ of radius $r>0$
is the set
\[
	N(V, r) = \{ z \in Z : d(z, V) < r \}.
\]

\subsection{Acknowledgements}
We thank Fran{\c{c}}ois Dahmani, Bruce Klein\-er, Marc Lackenby, Xiangdong Xie and a referee for helpful comments.

\section{Relatively hyperbolic groups and transversality}\label{sec-rel-hyp-defs}

In this section we define relatively hyperbolic groups and their (Bow\-ditch) boundaries.
We introduce the notion of a transversal embedding, and show that such embeddings
persist into the coned-off graph of a relatively hyperbolic group, or into suitable peripheral 
fillings of the same.

\subsection{Basic definitions}
There are many (equivalent) definitions of relatively hyperbolic groups.
We give one here in terms of actions on a cusped space.
First we define our model of a horoball.

\begin{definition}
	Suppose $\Gamma$ is a connected graph with vertex set $V$ and edge set $E$,
	where every edge has length one.
	Let $T$ be the strip $[0,1] \times [1, \infty)$ in the upper half-plane model of $\HH^2$.
	Glue a copy of $T$ to each edge in $E$ along $[0,1] \times \{1\}$,
	and identify the rays $\{v\} \times [1,\infty)$ for every $v \in V$.
	The resulting space with its path metric is 
	the \emph{horoball} $\cB(\Gamma)$.
\end{definition}

These horoballs are hyperbolic with boundary a single point.  
(See the discussion following \cite[Theorem 3.8]{Bow-99-rel-hyp}.)
Moreover, it is easy to estimate distances in horoballs.
\begin{lemma}\label{lem-horoball-dist}
	Suppose $\Gamma$ and $\cB(\Gamma)$ are defined as above.
	Let $d_\Gamma$ and $d_\cB$ denote the corresponding path metrics.
	Then for any distinct vertices $x, y \in \Gamma$, identified with $(x,1), (y,1) \in \cB(\Gamma)$, we have
	\[
		d_{\cB}(x,y) \approx_1 2 \log(d_\Gamma(x,y)).
	\]
\end{lemma}
\begin{proof}
	Any geodesic in $\cB(\Gamma)$ will project to the image of a geodesic in $\Gamma$,
	so it suffices
	to check the bound in the hyperbolic plane, for points $(0,1)$ and $(t,1)$, with $t \geq 1$.
	But then
	$d_\cB((0,1),(t,1)) = \arccosh(1+\frac{t^2}{2})$, and
	\[
		\left| \arccosh\left(1+\frac{t^2}{2}\right) - 2\log(t) \right|
	\]
	is bounded (by $1$) for $t \geq 1$.
\end{proof}

\begin{definition}\label{def-rel-hyp}
	Suppose $G$ is a finitely generated group, and $\cP = \{P_1, P_2, \ldots, P_n\}$
	a collection of finitely generated subgroups of $G$.
	Let $S$ be a finite generating set for $G$, so that $S \cap P_i$ generates
	$P_i$ for each $i=1, \ldots, n$.
	
	Let $\Gamma(G)=\Gamma(G,S)$ be the Cayley graph of $G$ with respect to $S$, with word metric $d_G$.
	Let $Y$ be the disjoint union of $\Gamma(G,S)$ and copies $\cB_{hP_i}$ of $\cB(\Gamma(P_i, S \cap P_i))$ 
	for each left coset $hP_i$ of each $P_i$.
	Let $X = X(G, \cP) = Y / \sim$, where for each left coset $h P_i$ and each $g \in h P_i$ 
	the equivalence relation $\sim$ identifies $g \in \Gamma(G, S)$ with $(g,1) \in \cB_{h P_i}$.
	We endow $X$ with the induced path metric $d$, which makes $(X, d)$ a proper, geodesic metric space. 
	
	We say that $(G, \cP)$ is \emph{relatively hyperbolic} if $(X, d)$ is Gromov hyperbolic,
	and call the members of $\cP$ \emph{peripheral subgroups}.
\end{definition}
This is equivalent to the other usual definitions of (strong) relative hyperbolicity;
see \cite{Bow-99-rel-hyp}, \cite[Theorem 3.25]{Gro-Man-08-dehn-rel-hyp} or \cite[Theorem 5.1]{Hr-relqconv}.
(Note that the horoballs of Bowditch and of Groves-Manning are quasi-isometric.)

Let $\cO$ be the collection of all disjoint (open) horoballs in $X$, that is,
the collection of all connected components of $X \setminus \Gamma(G,S)$.
Note that $G$ acts properly and isometrically on $X$,
cocompactly on $X \setminus \left( \bigcup_{O \in \cO} O \right)$,
and the stabilizers of the sets $O \in \cO$ are precisely the conjugates
of the peripheral subgroups. Subgroups of conjugates of peripheral subgroups are called \emph{parabolic subgroups}.

The \emph{boundary} of $(G, \cP)$ is the set $\bdry (G, \cP) = \bdry X(G, \cP) = \bdry X$.
Choose a basepoint $w \in X$, and denote the Gromov product in $X$ by $(\cdot|\cdot)=(\cdot|\cdot)_w$;
as $X$ is proper, this can be defined as $(a|b) = \inf \{d(w, \gamma)\}$, where the infimum is taken over all
(bi-infinite) geodesic lines $\gamma$ from $a$ to $b$; such a geodesic is denoted by $(a,b)$.

A \emph{visual metric} $\rho$ on $\bdry X$ with parameters $C_0, \eps>0$ is a metric so that
for all $a, b \in \bdry X$ we have
$e^{-\epsilon (a|b)}/C_0\leq \rho(a,b)\leq C_0 e^{-\epsilon (a|b)}$.
As $X$ is proper and geodesic, for every $\epsilon>0$ small enough there exists $C_0$
and a visual metric with parameters $C_0,\eps$, see e.g.\ \cite[Proposition 7.10]{Ghys-dlH-90-hyp-groups}.
We fix such choices of $\rho$, $\eps$ and $C_0$ for the rest of the paper.

For each $O\in\cO$ the set $\bdry O$ consists of a single point $a_O \in \bdry X$ called a \emph{parabolic point}.
We also set $d_O=d(w,O)$.

Horoballs can also be viewed as sub-level sets of Busemann functions.

\begin{definition}
	Given a point $c \in \bdry X$, and basepoint $w \in X$,   The \emph{Busemann function} corresponding to $c$ and $w$
	is the function $\beta_c( \cdot, w) : X \ra \R$ defined by
	\[
		\beta_c(x, w) = \sup_\gamma\left\{\limsup_{t \ra \infty} ( d(\gamma(t), x) - t )\right\},
	\]
where the supremum is taken over all geodesic rays $\gamma: [0, \infty) \ra X$
	from $w$ to $c$.
\end{definition}
The supremum in this definition is only needed to remove the dependence on the choice of $\gamma$, however 
for any two rays from $w$ to $c$ the difference 
between the corresponding $\limsup$ values is bounded as a function of the hyperbolicity 
constant only, see e.g. \cite[Lemma 8.1]{Ghys-dlH-90-hyp-groups}.

\begin{lemma}\label{horob}
 There exists $C = C(\eps, C_0, X)$ so that for each $O\in\cO$ we have
$$\{x\in X: \beta_{a_O}(x,w)\leq -d_O-C\}\subseteq O\subseteq \{x\in X: \beta_{a_O}(x,w)\leq -d_O+C\}.
$$
\end{lemma}
\begin{proof}
	One can argue directly, or note that $-\beta_{a_O}(\cdot, w)$ is a horofunction according to Bowditch's definition, and
	so the claim follows from the discussion before \cite[Lemma 5.4]{Bow-99-rel-hyp}.
\end{proof}

From now on, for any Gromov hyperbolic metric space $Y$ we denote by $\del_Y$ a hyperbolicity constant of $Y$, i.e.\ 
given any geodesic triangle with vertices in $Y\cup\partial Y$, each side of the triangle is contained in 
the union of the $\delta_Y$-neighbourhoods of the other two sides.

\begin{lemma}
\label{hor-qconv}
 Let $(G, \cP)$ be relatively hyperbolic, and $X=X(G, \cP)$ as above. 
 For any left coset $gP$ of some $P\in\cP$ and any $x,y\in gP$, all geodesics from $x$ to $y$ are contained in the corresponding $O\in\cO$ except for, at most, an initial and a final segment of length at most $2\delta_X+1$.
\end{lemma}

\begin{proof}
 From the definition of a horoball, any ``vertical'' ray 
 $\{v\}\times [1,\infty)$ in $O$ is a geodesic ray in both $X$ and $O$. 
 Also, for any $v \in gP$ and $t \geq 1$, the closest point in $X \setminus O$ to $(v,t) \in O$ 
 is $(v,1)\in \overline{O}$.
 
 Consider a geodesic triangle with sides $\{x\}\times [1,\infty)$, $\{y\}\times [1,\infty)$ 
 and a geodesic $\gamma$ in $X$ from $x$ to $y$. 
 If $p\in\gamma$ has distance greater than $2\del_X+1$ from both $x$ and $y$, 
 then it is $\del_X$ close to a point $q$ on one of the vertical rays, 
 which has distance at least $\del_X+1$ from $\{x,y\}$.
 As $d(q, X \setminus O) = d(q,\{x,y\})$, we have $d(p, X\setminus O) \geq 1$ and $p \in O$.
\end{proof}

Finally, we extend the distance estimate of Lemma~\ref{lem-horoball-dist} to $X$.
\begin{lemma}\label{lem-horoball-dist2}
	Let $(G, \cP)$ be relatively hyperbolic, and $X=X(G, \cP)$ as above.
	There exists $A=A(X) < \infty$ so that
	for any left coset $gP$ of some $P \in \cP$, with metric $d_P$, 
	for any distinct $x, y\in gP$, we have
	\[
		d(x,y) \approx_A 2 \log(d_P(x,y)).
	\]
\end{lemma}
\begin{proof}
	Consider a geodesic $\gamma$ from $x$ to $y$. 
	If $d(x,y) \leq 4\del_X+3$, then there is a uniform upper bound of $C = C(X)$ 
	on $d_P(x,y)$ as $d_P$ is a proper function with respect to $d$,
	so this case is done.
	Otherwise, by Lemma \ref{hor-qconv}, $\gamma$ has a subgeodesic $\gamma'$ connecting points $x', y'\in gP$,
	with $\gamma'$ entirely contained in the horoball $O$ that corresponds to $gP$, and
	$d(x, y) \approx_{4\del_X+2} d(x',y')$,
	so $d(x',y') \geq 1$.  In particular, $x' \neq y'$.
	
	As before, $d_P(x,x')$ and $d_P(y,y')$ are both at most $C$.
	Let $d_\cB$ denote the horoball distance for points in $\overline{O}$.
	Notice that $d(x',y') = d_\cB(x',y')$ as 
	$\gamma'$ is a geodesic in $X$ connecting $x'$ to $y'$ entirely contained in $\overline{O}$. 
	Using Lemma~\ref{lem-horoball-dist} we have
\begin{align*}
	d(x,y) & \approx_{4\del_X+2} d(x',y') = d_\cB(x',y')
		\\ & \approx_1 2 \log(d_P(x',y')) \approx_{C'} 2 \log(d_P(x,y)),
\end{align*}
for an appropriate constant $C'=C'(C)$.
\end{proof}

\subsection{Transversality and coned-off graphs}
\label{ssec-transverse-coned}

Our goal here is to define transversality of quasi-isometric embeddings, and
show that a transversal quasi-isometric embedding of $\HH^2$ in $\Gamma=\Gamma(G)$ will
persist if we cone-off the Cayley graph.
Recall that a function $f:X \ra X'$ between metric spaces is a \emph{$(\lambda,C)$-quasi-isometric embedding} (for some $\lambda\geq 1, C \geq 0$) if for all $x,y \in X$,
\[
\frac{1}{\lambda}d(x,y)-C \leq d(f(x),f(y)) \leq \lambda d(x,y)+C.
\]

We continue with the notation of Definition~\ref{def-rel-hyp}.

\begin{definition}\label{transversal}
Let $(G, \cP)$ be a relatively hyperbolic group.
Let $Z$ be a geodesic metric space.
 Given a function $\eta: [0, \infty) \ra [0, \infty)$,
 a quasi-isometric embedding $f:Z\to (G, d_G)$ is \emph{$\eta$-transversal} if 
 for all $M \geq 0$, $g \in G$, and $P \in \cP$, we have 
 $\diam(f(Z)\cap N(gP,M)) \leq \eta(M)$.
 
 A quasi-isometric embedding $f:Z \to G$ is \emph{transversal} if it is $\eta$-transversal for
 some such $\eta$.
\end{definition}

Let $\hat{\Gamma}$ be the coned-off graph of $\Gamma(G)$:
for every left coset $hP$ of every $P \in \cP$, add a vertex to $\Gamma$, and
add edges of length $1/2$ joining this vertex to each $g \in hP$.
Let $c:\Gamma\to\hat{\Gamma}$ be the natural inclusion.
This graph is hyperbolic, by the equivalent definition of relative hyperbolicity
in \cite{Farb-98-rel-hyp} (see also \cite[Definition 3.6]{Hr-relqconv}).
Recall that a \emph{$\lambda$-quasi-geodesic} is a $(\lambda, \lambda)$-quasi-isometric embedding of an interval (which need not be continuous).
\begin{lemma}
\label{tranqgeod}
 Let $(G, \cP)$ be a relatively hyperbolic group.
 If $\alpha: \R \ra (G, d_G)$ is an $\eta$-transversal $\lambda$-quasi-geodesic,
 then $c(\alpha)$ is a $\lambda'$-quasi-geodesic in $\hat{\Gamma}$, where $\lambda'=\lambda'(\lambda, \eta,G, \cP)$.
 Moreover, $\alpha$ is also a $\lambda'$-quasi-geodesic in $X(G,\cP)$.
\end{lemma}

\begin{proof}
  On $\Gamma$ we have $\frac{1}{C} d_{\hat{\Gamma}}\leq d_{X} \leq d_G$ for a suitable $C\geq 1$, and $\alpha:\R \ra (G, d_G)$ is
 a $(\lambda,\lambda)$-quasi-isometry.  Therefore, it suffices to show that
 for some $\lambda' = \lambda'(\lambda, \eta, \hat{\Gamma})$, for any $x_1, x_2 \in \alpha$, we have
 \begin{equation}\label{eq-trans-bound}
   d_{\hat{\Gamma}}(x_1, x_2) \geq \tfrac{1}{\lambda'} d_G(x_1, x_2) - \lambda'.
 \end{equation}
 
 Let $\gamma$ be the sub-quasi-geodesic of $\alpha$ with endpoints $x_1, x_2$.
 Let $\hat{\gamma}$ be a geodesic in $\hat{\Gamma}$ with endpoints $c(x_1), c(x_2)$.
 
 Now let $\pi : \hat{\Gamma} \ra \hat{\gamma}$ be a closest point projection map, fixing $x_1, x_2$.
 As $\hat{\Gamma}$ is hyperbolic, such a projection map is \emph{coarsely Lipschitz}:
 there exists $C_1 = C_1(\hat{\Gamma})$ so that for all $x,y \in \hat{\Gamma}$,
 $d_{\hat{\Gamma}}(\pi(x), \pi(y)) \leq C_1 d_{\hat{\Gamma}}(x, y) + C_1$.
 
 By \cite[Lemma 8.8]{Hr-relqconv}, there exists $C_2=C_2(G, \cP, \lambda)$ so that every vertex of 
 $\hat{\gamma}$ is at most a distance $C_2$ from $\gamma$ in $\Gamma$ (not just $\hat{\Gamma}$).
 Let $\pi': (\hat{\gamma}, d_{\hat{\Gamma}}) \ra (\gamma, d_G)$ be a map so that
 for all $x \in \hat{\gamma}$, $d_G(\pi'(x), x) \leq C_2$,
 and assume that $\pi'$ fixes $x_1, x_2$.
 This map is coarsely Lipschitz also.  It suffices to check this for vertices
 $x,y \in \hat{\gamma} \cap \Gamma$ with $d_{\hat{\Gamma}}(x,y) =1$.
 If $x$ and $y$ are connected by an edge of $\Gamma$, then clearly $d_G(\pi'(x),\pi'(y))\leq 2C_2+1$.
 Otherwise, $\pi'(x), \pi'(y)$ are both in $\gamma$ and within distance $C_2$ of the same left coset of
 a peripheral group, so by transversality, $d_G(\pi'(x), \pi'(y)) \leq \eta(C_2)$.
 
 Thus, for $C_3 = \max\{2C_2+1, \eta(C_2)\}$, \eqref{eq-trans-bound} with $\lambda' = C_3 C_1$ follows from
 \begin{align*}
 	d_G(x_1, x_2) & = d_G(\pi' \circ \pi(x_1), \pi' \circ \pi(x_2))
		\leq C_3 d_{\hat{\Gamma}}(\pi(x_1), \pi(x_2)) \\
		& \leq C_3 (C_1 d_{\hat{\Gamma}}(x_1, x_2) + C_1). \qedhere
 \end{align*}
\end{proof}

\begin{proposition}\label{coned-off}
 Suppose $Z$ is a geodesic metric space, and $(G, \cP)$ a relatively hyperbolic group.
 If a quasi-isometric embedding $f:Z\to G$ is transversal then $c\circ f: Z\to \hat{\Gamma}$ is a quasi-isometric embedding, quantitatively.
\end{proposition}
\begin{proof}
By Lemma \ref{tranqgeod}, whenever $\gamma$ is a geodesic in $Z$, $c\circ f(\gamma)$ is a quasi-geodesic with uniformly bounded constants in $\hat{\Gamma}$.
\end{proof}

\subsection{Stability under peripheral fillings}
\label{ssec-periph-fill}
We now consider peripheral fillings of $(G, \{P_1, \ldots, P_n\})$.
The results here are not used in the remainder of the paper.

Suppose $N_i\triangleleft P_i$ are normal subgroups. 
The \emph{peripheral filling} of $G$ with respect to $\{N_i\}$ is defined as
$$G(\{N_i\})=G/\ll \bigcup_i N_i\gg .$$
(See \cite{Osin-fill} and \cite[Section 7]{Gro-Man-08-dehn-rel-hyp}.)
Let $p:G\to G(\{N_i\})$ be the quotient map.

\begin{proposition}\label{perfill}
 Let $(G, \cP)$ be a relatively hyperbolic group, and $G(\{N_i\})$ a peripheral filling of $G$, as
 defined above.
 
 Let $Z$ be a geodesic metric space, and
 suppose that $f: Z \ra G$ is an $\eta$-transversal $(\lambda,\lambda)$-quasi-isometric embedding.
 There exists $R_0 = R_0(\eta, \lambda, G, \cP)$ so that if $B(e,R_0)\cap N_i=\{e\}$ for each $i$,
 then $p\circ f:Z\to G(\{N_i\})$ is a quasi-isometric embedding.
\end{proposition}
\begin{proof}
 We will use results from \cite{DGO-hypembrot} (see also \cite{DelGr-courbure,Co-rot}).
 For any sufficiently large $R_0$ we combine Propositions 7.7 and 5.28 in \cite{DGO-hypembrot} to show that
 $G(\{N_i\})$ acts on a certain $\delta'$-hyperbolic space $Y$ (namely $Y=\mathbb{X}'/Rot$ for $\mathbb{X}'$ and $Rot$ arising from $\mathcal{R}$
 in the notation of \cite{DGO-hypembrot}), with the following properties:
\begin{enumerate}
 \item $\delta'$ only depends on the hyperbolicity constant of $X = X(G, \cP)$,
 \item there is a $1$-Lipschitz map $\psi:G(\{N_i\})\to Y$ (this follows from the construction
	of $Y$ and the fact that the inclusion of $G$ in $X$ is $1$-Lipschitz),
 \item there is a map $\phi:(N(G,L)\subseteq X)\to Y$ such that, for each $g\in G$, 
	$\phi|_{B(g,L)}$ is an isometry, where $L=L(R_0, X)\to\infty$ as $R_0\to\infty$.
 \item $\psi\circ p=\phi\circ\iota$, where $\iota:G\to X$ is the natural inclusion.
\end{enumerate}
Let $\gamma$ be any geodesic in $Z$. 
By Lemma~\ref{tranqgeod}, $\iota \circ f(\gamma)$ is a $\lambda'$-quasi-geodesic in $X$, for $\lambda' = \lambda'(\lambda, \eta, G, \cP)$.
Let $\alpha$ be a geodesic in $X$ connecting the endpoints of $\iota \circ f(\gamma)$.
Let $C_1 = C_1(\delta', \lambda')$ bound the distance between each point on $\alpha$ and $G\subseteq X$.

Suppose that $L$ as in $(3)$ satisfies $L\geq\ C_1+8\delta'+1$. Then for each $x\in\alpha$ we have that $\phi|_{B(x,8\delta'+1)}$
is an isometry, and so \cite[Theorem III.H.1.13-(3)]{BH-99-Metric-spaces} gives that $\phi(\alpha)$ is a $C_2$-quasi-geodesic,
where $C_2 = C_2(\delta')$.
This implies that $(\phi\circ\iota\circ f)(\gamma)$ is a $C_3$-quasi-geodesic, with $C_3 = C_3(C_1, C_2)$.
Let $x_1, x_2$ be the endpoints of $\gamma$.
Using $(2)$ and $(4)$ above, we see that
\begin{align*}
	 d_{G(\{N_i\})}( p\circ f(x_1), p\circ f(x_2) ) & \geq 
		d_Y( \psi \circ p\circ f(x_1), \psi \circ p\circ f(x_2) ) \\ 
		& = d_Y( \phi\circ\iota\circ f (x_1), \phi\circ\iota\circ f (x_2) ) \\
		& \geq \tfrac{1}{C_3} d_Z(x_1, x_2) - C_3.
\end{align*}
On the other hand, recall that $p$ is $1$-Lipschitz, so
\begin{equation*}
 d_{G(\{N_i\})}( p\circ f(x_1), p\circ f(x_2) ) \leq d_G (f(x_1), f(x_2)) \leq \lambda d_Z(x_1, x_2) + \lambda.
\end{equation*}
As $\gamma$ was arbitrary, we are done.
\end{proof}

As discussed in the introduction, 
we can use Proposition \ref{perfill} to find interesting examples of relatively hyperbolic groups 
with quasi-isometrically embedded copies of $\HH^2$, but whose peripheral groups are not virtually nilpotent.
We note the following lemma.

\begin{lemma}\label{lem-free-quotients}
	Let $F_4$ be the free group with four generators, and let $R$ be fixed.
	Then there are normal subgroups $\{K_\alpha\}_{\alpha \in \R}$ of $F_4$ so that
	the quotient groups $F_4/K_\alpha$ are amenable but not virtually nilpotent,
	so that if $\alpha \neq \beta$ then $F_4/K_\alpha$ and $F_4/K_\beta$ are not quasi-isometric,
	and so that $K_\alpha \cap B(e,R) = \{e\}$.	
\end{lemma}
\begin{proof}
It is shown in \cite{Gri-uncountgrowth} that there is an uncountable family of $4$-generated groups 
$\{F_4/K'_\alpha\}_{\alpha\in\R}$ of intermediate growth with distinct growth rates. In particular,
these groups are amenable, not virtually nilpotent and pairwise non-quasi-isometric. 
To conclude the proof, let $K$ be a finite index normal subgroup of $F_4$ so that $K \cap B(e,R) = \{e\}$,
and let $K_\alpha = K'_\alpha \cap K \vartriangleleft F_4$.
As $F_4/K_\alpha$ is a finite extension of $F_4/K'_\alpha$, it inherits all the properties above.
\end{proof}

\begin{example}\label{nonvnilpquot}
	Let $M$ be a closed hyperbolic $3$-manifold so that $G= \pi_1(M)$ is hyperbolic relative to a subgroup
	$P \leq G$ that is isomorphic to $F_4$.
	For example, let $M'$ be a compact hyperbolic manifold whose boundary $\partial M'$ is a totally geodesic
	surface of genus $3$, and let $M$ be the double of $M'$ along $\partial M'$.
	Observe that $\pi_1(M)$ is hyperbolic relative to $\pi_1(\partial M')$, and $\pi_1(\partial M')$ is hyperbolic
	relative to a copy of $\pi_1(S')=F_4$, where $S' \subset \partial M'$ is a punctured genus $2$ subsurface.
	Thus $\pi_1(M)$ is hyperbolic relative to $\pi_1(S')$.
	
	Since $G$ is hyperbolic with $2$-sphere boundary, and $P$ is quasi-convex in $G$ with Cantor set boundary (Lemma~\ref{lem-x-asymp-tree-graded}(2)),
	the hypotheses of Theorem \ref{thm-main-hyp-rel-hyp} are satisfied for $\cP_1 = \emptyset$ and $\cP_2 = \{P\}$.
	Therefore, we find a transversal quasi-isometric embedding of $\HH^2$ into $G$.
	
	Let $R_0$ be chosen by Proposition~\ref{perfill}.
	As $P$ is quasi-convex in $G$, we choose $R$ so that for $x \in P$, if $d_P(e, x) \geq R$ then $d_G(e, x) \geq R_0$.
	Now let $\{K_\alpha\}$ be the subgroups constructed in Lemma \ref{lem-free-quotients}.
	By Proposition~\ref{perfill}, for each $\alpha \in \R$ the peripheral filling
	$G_\alpha=G/\ll K_\alpha \gg$ is 
	relatively hyperbolic and contains a quasi-isometrically embedded copy of $\HH^2$. 
	
	As $P/K_\alpha$ is non-virtually cyclic and amenable, it does not have a non-trivial relatively hyperbolic structure.
	Therefore, $G_\alpha$ is not hyperbolic relative to virtually nilpotent subgroups, for in any peripheral
	structure $\cP$, some peripheral group $H \in \cP$ must be quasi-isometric to $P/K_\alpha$ by \cite[Theorem 4.8]{BDM-thick}.
	
	Finally, if $\alpha\neq \beta$ then $G_\alpha$ is not quasi-isometric to $G_\beta$ by \cite[Theorem 4.8]{BDM-thick}
	as $P/K_\alpha$ and $P/K_\beta$ are not quasi-isometric.
\end{example}

\section{Separation of parabolic points and horoballs}
\label{sec-separation-in-bdry}

In this section we study how the boundaries of peripheral subgroups are separated in $\bdry X$.
We also establish a preliminary result on quasi-isometrically embedding copies of $\HH^2$.

\subsection{Separation estimates}
We begin with the following lemma.

\begin{lemma}\label{lem-x-asymp-tree-graded}
Let $(G, \cP_1)$ and $(G, \cP_1 \cup \cP_2)$ be relatively hyperbolic groups,
where all peripheral subgroups in $\cP_2$ are hyperbolic groups ($\cP_2$ is allowed to be empty),
and set $X=X(G,\cP_1)$.
Let $\cH$ denote the collection of the horoballs of $X$ and the left cosets
of the subgroups in $\cP_2$; more precisely, the images of those left cosets under the natural inclusion $G\hookrightarrow X$.
For $H \in \cH$, let $d_H = d(w, H)$.

Then the collection of subsets $\cH$ has the following properties.
\begin{enumerate}
 \item For each $r\geq 0$ there is a uniform bound $b(r)$ on $\diam(N(H,r)\cap N(H',r))$ for each $H,H'\in\cH$ with $H\neq H'$.
 \item Each $H \in \cH$ is uniformly quasi-convex in $X$.
 \item There exists $R$ such that, given any $H \in \cH$ and any geodesic ray $\gamma$ connecting
 $w$ to $a\in \bdry H$, the subray of $\gamma$ whose
 starting point has distance $d_H$ from $w$ is entirely contained in $N(H,R)$.

\end{enumerate}
\end{lemma}
\begin{proof}
In this proof we use results from \cite{DSp-05-asymp-cones}, which uses the equivalent
definition that $(G, \cP)$ is relatively hyperbolic if and only 
the Cayley graph of $G$ is ``asymptotically tree graded'' with respect to the collection
of left cosets of groups in $\cP$ \cite[Definition 5.9, Theorem 8.5]{DSp-05-asymp-cones}.

We first show $(1)$.
Let $gP, g'P'$ be the left cosets corresponding to $H, H'$.
As $G$ acts properly on $X$, given $r$ there exists $r'$ so that 
$N(H,r) \cap N(H',r)$ (using the metric $d$ on $X$)
is contained in the $r'$-neighbourhood in $X$ of
$N_G(gP,r') \cap N_G(g'P',r') \subset G \subset X$ (where $N_G$ indicates that we use the
metric $d_G$).
The diameter of $N_G(gP,r') \cap N_G(g'P',r')$ has a uniform bound in
$G$ by~\cite[Theorem 4.1($\alpha_1$)]{DSp-05-asymp-cones}; as $d \leq d_G$
it is also uniformly bounded in $X$.

Let us show $(2)$.
Uniform quasi-convexity of the horoballs is a consequence of Lemma \ref{hor-qconv}. 
If $H$ is a left coset of a peripheral subgroup in $\cP_2$, then it is quasi-convex in the Cayley graph 
$\Gamma$ of $G$ \cite[Lemma 4.15]{DSp-05-asymp-cones}. 
What is more, by \cite[Theorem 4.1($\alpha_1$)]{DSp-05-asymp-cones}, geodesics in $\Gamma$ connecting points 
of $H$ are transversal with respect to $\cP_1$. 
Therefore, by Lemma \ref{tranqgeod} they are quasi-geodesics in $X$. 
We conclude that $H$ is quasi-convex in $X$ since pairs of points of $H$ can be joined by quasi-geodesics 
(with uniformly bounded constants) which stay uniformly close to $H$.

We now show $(3)$.
Choose $p \in H$ so that $d(w,p)\leq d_H+1$.
As $a \in \bdry  H$ and $H$ is quasi-convex, there exists $C=C(X)$
so that a geodesic ray $\gam'$ from $p$ to $a$ lies in the $C$-neighbourhood of $H$.

Choose $x$ on a geodesic $\gam''$ from $w$ to $p$ so that $d(w,x) = d_H-C-\del_X-1$.
We must have $d(x,\gam') > \del_X$, otherwise $d(w,H) \leq d(w,x)+d(x,\gam')+C < d_H$.
As the geodesic triangle with sides $\gam, \gam', \gam''$ is $\del_X$ thin, 
$x$ is within $d_X$ of a point $y \in \gam$.
Note that $d(y,H) \leq C+2\del_X+1$, so by $(2)$
the subray of $\gam$ starting at $y$ lies in a uniformly bounded 
neighbourhood of $H$.
As $d(w,y) \leq d_H -C -1 < d_H$, we are done.
\end{proof}

From this lemma, we can deduce separation properties for the boundaries of sets in $\cH$.
\begin{lemma}\label{lem-parabolic-points}
We make the assumptions of Lemma~\ref{lem-x-asymp-tree-graded}.
Then there exists $C=C(X)$ so that for each $H,H'\in\cH$ with $H\neq H'$ and $d_H\geq d_{H'}$ we have
$$\rho(\bdry H,\bdry H')\geq e^{-\epsilon d_H}/C. $$
\end{lemma}

\begin{proof}
 Let $a\in \bdry H, a'\in \bdry H'$. 
 We have to show that $(a|a')\lesssim d_H$. 
 Let $\gamma$, $\gamma'$ be rays connecting $w$ to $a, a'$, respectively.
 For each $p\in\gamma$ such 
 that $d(w,p)\leq (a|a')$ there exists $p'\in\gamma'$ such that $d(p',w)=d(p,w)$ and $d(p,p')\leq C_1=C_1(\delta_X)$.
 With $b(r)$ and $R$ as found by Lemma~\ref{lem-x-asymp-tree-graded}, set $C_2=b(R+C_1)$. 
 
 Suppose that $(a|a')\geq d_H+C_2+1$. Consider $p\in\gamma$, $p' \in \gamma$ such 
 that $d(p,w) = d(p',w)=d_H$.  
 By Lemma~\ref{lem-x-asymp-tree-graded}(3), we have $p\in N(H,R)$ and $p'\in N(H',R)$, as $d_H\geq d_{H'}$.
 So, $p\in N(H,R+C_1)\cap N(H',R+C_1)$.  The same holds also when $p\in\gamma$ is such that 
 $d(p,w)=d_H+C_2+1$. Therefore we have $\diam(N(H,R+C_1)\cap N(H',R+C_1))> C_2$, contradicting Lemma~\ref{lem-x-asymp-tree-graded}(1). 
 Hence $(a|a')< d_H+C_2+1$, and we are done.
\end{proof}

Conversely, we show that separation properties of certain points in the boundary $\bdry X$ have
implications for the intersection of sets in $X$.
\begin{lemma}\label{farfrombpoints}
We make the assumptions of Lemma~\ref{lem-x-asymp-tree-graded}.
 
 Let $\gam$ be a geodesic line connecting $w$ to $a\in\bdry X$. Suppose that for some $H\in\cH$ and
$r\geq 1$ we have $\rho(a,\bdry H)\geq e^{-\epsilon d_H}/r$. Then $\gamma$ intersects $N(H,L)$ in a
set of diameter bounded by $K+2L$, for $K=K(r, X)$ and any $L \geq 0$.
\end{lemma}

\begin{proof}[Proof of Lemma~\ref{farfrombpoints}]
We will treat the horoball case and the left coset case separately, beginning with the latter.
\par
 We can assume that $H$ is not bounded, for otherwise the lemma is trivially true. 
 Due to quasi-convexity, each point on $H$ is at uniformly bounded distance from a geodesic
 line connecting points $a_1, a_2$ in $\bdry H$, and thus also at uniformly bounded distance, say $C_1=C_1(X)$,
 from either a ray connecting $w$ to $a_1$ or a ray connecting $w$ to $a_2$.
 
 Let $c\in \bdry H$ and let $\gamma'$ be a ray connecting $w$ to $c$. 
 As $\rho(a,c) \geq e^{-\eps d_H}/r$, we have that 
 $(a|c)\leq d_H+r'+C_2$, for $r'=\log(r)/\epsilon$ and $C_2=C_2(\delta_X, \eps, C_0)$.
 
 There exists $C_3 = C_3(C_1, X)$ so that any point $x$ on $\gamma$ such that
 $d(x,w) \geq (a|c)+L+C_3$ has the property that $d(x,\gamma')\geq L+C_1$.
 This applies to all rays connecting $w$ to some $c\in\bdry H$, 
 and so $d(x, H) > L$.
 
 Also, if $x\in\gamma$ and $d(x,w)<d_H-L$ then clearly $d(x,H)>L$.
 Thus if $\gamma\cap N(H,L) \neq \emptyset$ we have $(a|c)+L+C_3\geq d_H-L$, and
$$\diam(\gamma\cap N(H,L))\leq (a|c)+L+C_3 - (d_H-L) \lesssim_{C_2+C_3} r'+2L.$$
\par
We are left to deal with the horoball case. Let $c$ and $\gamma'$ be as above. 
Once again, $(a|c)\leq d_H+r'+C_2$, for $r'=\log(r)/\epsilon$. 
\begin{figure}[h]
\begin{center}
\includegraphics[scale=0.5]{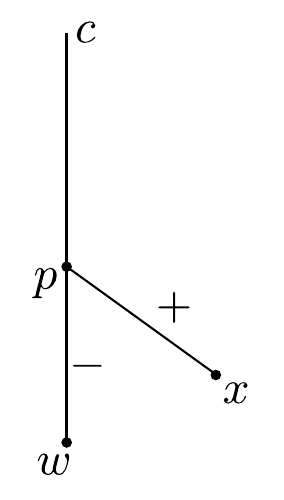}
\caption{How to compute the Busemann function in a tree.}
\end{center}
\end{figure}
By Lemma~\ref{horob}, if $\beta_c(x,w) > -d_H+C$ then $x \notin H$.
As $\beta_c$ is a $1$-Lipschitz function, if $\beta_c(x,w) > -d_H+C+L$, then
$x \notin N(H,L)$.

Given $x \in \gam$ with $d(x,w) \geq (a|c)$,
let $T'$ be the union of $\gam$ and the segment of $\gam'$ between $w$ and $x$.
Consider an approximating tree $T$ for $T'$
(see Figure 1), where $d(w,p) = (a|c)$ and the length of $d(w,x)$ is preserved.
By hyperbolicity, there is a $(1,C_3)$-quasi-isometric
map from $T'$ to $T$ where $C_3 = C_3(\del_X)$.
In $T$, 
\[ \beta_c(x,w)=d(x,p)-d(w,p) = d(x,w)-2d(p,w) = d(x,w) - 2(a|c). \]
This means that in $X$, there exists $C_4 = C_4(C_3)$ so that if $d(x,w) \geq (a|c)$,
\[
	\beta_c(x,w) > d(x,w)-2(a|c)-C_4.
\]
Thus, if $d(x,w) > 2 (a|c)-d_H+C+L + C_4$, we have $\beta_c(x,w) > -d_H+C+L$, so
$x \notin N(H,L)$.
Arguing as before, one sees that if $\gamma \cap N(H,L) \neq \emptyset$, for $C_5=C+2C_2+C_4$,
\[
	\diam(\gamma \cap N_L(H))\leq (2(a|c)-d_H+C+L+C_4) - (d_H-L) 
	\lesssim_{C_5} 2r'+2L.\qedhere
\]
\end{proof}

\subsection{Embedded planes}

In order to find a quasi-isometrically embedded copy of $\HH^2$
in a relatively hyperbolic group, we only need to embed a half-space of $\HH^2$ into our model space $X$,
provided that the embedding does not go too far into the horoballs.  (Compare with \cite{BK-05-quasiarc-planes}.)
As we see later, this means that we do not need to embed a quasi-circle into the boundary of $X$, but merely a quasi-arc.

\begin{definition}
	The \emph{standard half-space} in $\HH^2$ is the set $Q = \{ (x,y) : x^2+y^2 <1, x \geq 0 \}$
	in the Poincar\'e disk model for $\HH^2$.
\end{definition}

Let $G$, $\cP_1$, $\cP_2$, $\cH$ and $X=X(G, \cP_1)$ be as in Lemma~\ref{lem-x-asymp-tree-graded}.

\begin{proposition}\label{prop-avoid-horoballs}
 Let $f:Q\to X$ be a $(\mu,\mu)$-quasi-isometric embedding of the standard half-space $\HH^2$ into $X$,
 with $f((0,0))=w$. 
 Suppose there exists $\lambda > 0$ so that for each $a\in\bdry Q$, $H\in\cH$ we have
$$d(f(a),\bdry H)\geq e^{-\epsilon d_H}/\lambda.$$
Then there exists a transversal (with respect to $\cP = \cP_1 \cup \cP_2$)
	quasi-isometric embedding $g:\mathbb{H}^2\to \Gamma(G)$.
\end{proposition}

\begin{proof}
Each point $x \in Q \setminus \{(0,0)\}$ lies on a unique geodesic $\gamma_x$ connecting $(0,0)$ to a point
$a_x$ in $\bdry Q$. As $f(\gamma_x)$ is a $\mu$-quasi-geodesic and $X$
is hyperbolic, $f(\gamma_x)$ lies within distance $C_1=C_1(\mu, \delta_X)$ 
from a geodesic $\gamma_x'$ from $w$ to $f(a_x)$.
Let $x' \in \gam_x'$ satisfy $d(f(x),x') \leq C_1$.

Given two such points $x, y \in Q \setminus \{(0,0)\}$, let $x'',y''$ be the points
on $\gam_x', \gam_y'$ at distance $(x'|y')$ from $w$.
By hyperbolicity, $d(x'',y'') \leq C$, and and $x'', y''$ both lie within $C$ of 
any geodesic $\gam_{x'y'}$ from $x'$ to $y'$, for $C = C(\del_X)$.

If $f(x)$ and $f(y)$ lie in $N(H,L)$, for some $H \in \cH$, $L \geq 0$,
then $\gam_{x'y'}$ lies in $N(H,L')$, for some $L'=L'(L,\del_X,\cH,C_1)$,
by the quasiconvexity of $H$ (Lemma~\ref{lem-x-asymp-tree-graded}(2)).
Thus $x',x'',y',y''$ are in $N(H,L'+C)$.
By Lemma~\ref{farfrombpoints}, $d(x'',x')$ and $d(y'',y')$ are both at most
$C' = K+2(L'+C)$, for $K = K(\lambda,X)$.
Thus $d(f(x),f(y)) \approx_{2C_1+C} d(x'',x')+d(y',y'') \leq 2C'$, 
so we have the bound
\begin{equation}\label{eqhalfplane}
	\diam(N(H,L) \cap f(Q)) \leq 2C' + 2C_1+C.
\end{equation}
In particular, any point on $\gamma_x'$ is $C_2=C_2(\lambda, X)$ close to a point in $\Gamma(G)$, 
and therefore any point in $f(Q)$ is $C_1+C_2$ close to a point in $\Gamma(G)$.

Notice that $Q$ contains balls $\{B_n\}$ of $\mathbb{H}^2$ of arbitrarily large radius, 
each of which admits a $(\mu,\mu)$-quasi-isometric embedding $f_n:B_n\to X$
so that each point in $f_n(B_n)$ is a distance $C_1+C_2$ close to a point in $\Gamma(G)$.
In particular, translating those embeddings appropriately using the action of $G$ on $X$ we can and do assume that the
center of $B_n$ is mapped at uniformly bounded distance from $w$. As $X$ is proper, we can use Arzel\`a-Ascoli 
to obtain a $(\mu',\mu')$-quasi-isometric embedding $\hat{g} : \HH^2 \ra X$ as the limit of a subsequence of 
$\{f_n\}$ (more precisely $\{f_n|_N\}$, where $N$ is a maximal $1$-separated net in $\HH^2$),
for $\mu' = \mu'(\mu, C_1, C_2)$. 

 We now define $g:\HH^2 \ra \Gamma(G)$ so that for each $x\in \mathbb{H}^2$ 
	we have $d(\hat{g}(x),g(x))\leq C_3$, for $C_3 = C_3(C_1, C_2)$.
As $\cH$ is invariant under the action of $G$, and $d_G$ is a proper function of $d$, 
$g$ is transversal by \eqref{eqhalfplane}.
	
It remains to show that $g$ is a quasi-isometric embedding.
	Pick $x,y\in \mathbb{H}^2$. Notice that
$$d_G(g(x),g(y))\geq d(g(x),g(y))\gtrsim_{2C_3} d(\hat{g}(x),\hat{g}(y))  \gtrsim_{\mu'} d(x,y)/\mu',$$
so it suffices to show that for some $\mu''$,
$$d_G(g(x),g(y))\leq \mu'' d(\hat{g}(x),\hat{g}(y))+\mu''.$$
Let $\gam$ be
the geodesic in $\mathbb{H}^2$ connecting $x$ to $y$. 
Let $\gam'$ be the piecewise geodesic in $X$ from 
$g(x)$ to $\hat{g}(x)$ to $\hat{g}(y)$ to $g(y)$, 
which is at Hausdorff distance at most $C_4=C_4(\mu', C_3, \delta_X)$
from $g(\gamma)$. 

Each maximal subpath $\beta \subseteq \gamma'$ contained in some horoball $O \in \cO$ has length $l(\beta)$ at most $C_5=C_5(C_4, X)$ by transversality \eqref{eqhalfplane}.
If $z, z' \in G$ are the endpoints of $\beta$, then $d_G(z,z') \leq M d(z,z')$, for some $M = M(C_5, X) \geq 1$,
as $d_G$ is a proper function of $d$, and if $z \neq z'$ then $d(z,z')$ is uniformly bounded away from zero.
So we can substitute $\beta$ by a subpath in $\Gamma(G)$ of length at most $M l(\beta)$.

Let $\alpha$ be the path in $\Gamma(G)$ obtained from $\gamma'$ by substituting each such $\beta$ in this way.
Clearly we have $l(\alpha)\leq M l(\gamma')$, and so
\[ d_G({g}(x),{g}(y))\leq l(\alpha)\leq M (d(\hat{g}(x),\hat{g}(y))+2C_3). \qedhere \]
\end{proof}

\subsection{The Bowditch space is visual}\label{ssec-visual}

In order for the boundary of a Gromov hyperbolic space to control the geometry of the space itself,
we require the following standard property.
\begin{definition}\label{def-visual}
	A proper, geodesic, Gromov hyperbolic space $X$ is \emph{visual} if there exists $w \in X$ and $C>0$
	so that for every $x \in X$ there exists a geodesic ray $\gamma:[0,\infty) \ra X$, with $\gamma(0)=w$ and 
	$d(x, \gamma) \leq C$.
\end{definition}
A weaker version of this condition, suitable for spaces that are not proper, or not geodesic, is given in 
\cite[Section 5]{BS-00-gro-hyp-embed}.

We record the following observation for completeness.  
\begin{proposition}\label{prop-visual}
	If $(G, \cP)$ is a relatively hyperbolic group with every $P \in \cP$ a proper subgroup of $G$,
	then $X(G, \cP)$ is visual.
\end{proposition}
\begin{proof}
	Let $w \in X = X(G, \cP)$ be the point corresponding to $e \in G$.  Let $x \in X$ be arbitrary.
	
	First we assume that $\cP \neq \emptyset$.
	The point $x$ lies in $N(O,1)$, for some (possibly many) $O \in \cO$.
	Let $a_O = \bdry O$ be the parabolic point corresponding to $O$, and let $b \in \bdry X \setminus\{a_O\}$ be any other point.
	Such a point exists as the peripheral group corresponding to $O$ has infinite index in $G$.
	
	As $X$ is proper and geodesic, there is a bi-infinite geodesic line $\gamma:(-\infty, \infty) \ra X$ 
	with endpoints $a_O$ and $b$.
	The parabolic group corresponding to $O$ acts on $X$, stabilising $a_O \in \bdry X$, 
	so that some translate $\gamma'$ of $\gamma$ is at distance at most $C$ from the point $x$.  (We can take $C=2$.)
	
	Denote the endpoints of $\gamma'$ by $a_O$ and $b'$, and let $\alpha$ be the geodesic ray from $w$ to $a_O$,
	and $\beta$ the geodesic ray from $w$ to $b'$.  As the geodesic triangle $\gamma', \alpha, \beta$  
	is $\delta_X$-thin, $x$ lies within a distance of $C+\delta_X$ of one of the geodesic rays $\alpha$ and $\beta$, and we are done.
	
	Secondly, if $\cP = \emptyset$, we have that $X(G, \emptyset)$ is a Cayley graph of $G$.
	Fix any $a \neq b$ in $\bdry X$.  As $X$ is proper and geodesic, there is a bi-infinite geodesic $\gamma$ from $a$ to $b$.
	As the action of $G$ on $X$ is cocompact, some translate of $\gamma$ passes within a uniformly bounded distance of $x$,
	and the proof proceeds as in the first case.
\end{proof}

\section{Boundaries of relatively hyperbolic groups}\label{sec-boundaries}

We now begin our study of the geometry of the boundary of a relatively hyperbolic group,
endowed with a visual metric $\rho$ as in Section~\ref{sec-rel-hyp-defs}.
In this section, we study the properties of being doubling and having partial self-similarity.

First, we summarize some known results about the topology of such boundaries.

\begin{theorem}[Bowditch]\label{thm-bowditch}
	Suppose $(G, \cP)$ is a one-ended relatively hyperbolic group which does not split
	over a subgroup of a conjugate of some $P \in \cP$, 
	and every group in $\cP$ is finitely presented, one or two ended, and contains no infinite torsion subgroup.
	Then $\bdry(G, \cP)$ is connected, locally connected and has no global cut points.
\end{theorem}
\begin{proof}
	Connectedness and local connectedness follow from \cite[Pro\-position 10.1]{Bow-99-rel-hyp} and \cite[Theorem 0.1]{Bow-MathZ}.
	
	Any global cut point must be a parabolic point \cite[Theorem 0.2]{Bow-99-conn-lim-set},
	but the splitting hypothesis ensures that these are not global cut points either
	\cite[Proposition 5.1, Theorem 1.2]{Bow-MathZ}.
\end{proof}

Recall that a point $p$ in a connected, locally connected, metrisable topological space
$Z$ is \emph{not a local cut point} if for every connected
neighbourhood $U$ of $p$, the set $U \setminus \{p\}$ is also connected.  
If, in addition, $Z$ is compact, then $Z$ is locally path connected, so
$p$ is not a local cut point if and only if every neighbourhood $U$ of $p$ contains an open $V$ with $p \in V \subset U$
and $V \setminus \{p\}$ path connected.

More generally, we have the following definition, used in the statement of Theorem~\ref{thm-main-hyp-rel-hyp}.
\begin{definition}\label{def-loc-disconn}
	A closed set $V \subsetneq Z$ in a connected, locally connected, metrisable topological space
	$Z$ \emph{does not locally disconnect} $Z$ if 
	for any open connected $U \subset Z$, the set $U \setminus V$ is also connected.
\end{definition}
For relatively hyperbolic groups, we note the following.
\begin{proposition}\label{prop-bdry-nice}
	Suppose $(G, \cP)$ is relatively hyperbolic with connected and locally connected boundary.
	Let $p$ be a parabolic point in $\bdry(G,\cP)$
	which is not a global cut-point. 
	Then $p$ is a local cut point if and only if the corresponding peripheral group has more than one end.
\end{proposition}
\begin{proof}
 The lemma follows, similarly to the proof of \cite[Proposition 3.3]{Dah-par-groups}, from the fact that the 
 parabolic subgroup $P$ corresponding to $p$ acts properly discontinuously and cocompactly on $\bdry(G,\cP)\backslash\{p\}$, 
 which is connected and locally connected. Let us make this precise. 
 
 Choose an open set $K_0$ with compact closure in $\bdry(G, \cP)\backslash \{p\}$, so that $PK_0=\bdry(G,\cP)\backslash\{p\}$.
 Then define $K_1$ as the union of all $qK_0$ for $q\in P$ with $d_P(q,e)\leq 1$.
 As $\bdry(G,\cP)\backslash\{p\}$ is connected and locally path connected, and $\overline{K_1}$ is compact, one can easily
 find an open, path connected $K$ so that $K_1 \subset K \subset \overline{K} \subset \bdry(G,\cP)\backslash\{p\}$.
 
 Now suppose that $P$ has one end.  Let $U$ be a neighbourhood of $p$.
 As $P$ acts properly discontinuously on $\bdry(G,\cP)\backslash\{p\}$, there exists $R$ so that
 if $d_P(e, g) > R$, then $gK \subset U$.  Let $Q$ be the unbounded connected component of $P\setminus B(e,R)$.
 Then $QK$ is path connected as for $g,h \in P$, if $d_P(g,h) \leq 1$, $gK \cap hK \neq \emptyset$.
 Finally, observe that $V = QK \cup \{p\} \subset U$ is a neighbourhood of $p$ so that
 $V \setminus \{p\} = QK$ is connected.
 
 Conversely, suppose that $p$ is not a local cut-point.
 Let $D$ be so that if $qK\cap rK\neq\emptyset$ then $d_P(q,r)\leq D$.
 Suppose we are given $R>0$.
 We can find a connected neighbourhood $U$ of $p$ in $\bdry(G,\cP)$ so that 
 $U \setminus \{p\}$ is path connected and $gK\cap U=\emptyset$ for all $g \in B(e, R+D) \subset P$.
 Let $R' \geq R+D$ be chosen so that $gK \cap U \neq \emptyset$ for all $g \in P \setminus B(e, R')$.
 Given $g,h\in P\backslash B(e,R')$ we can find a path in $U\backslash \{p\}$ connecting $gK$ to $hK$. 
 So, there exists a sequence $g=g_0, g_1, \ldots, g_n=h$ in $P \setminus B(e, R+D)$ 
 so that $g_iK\cap g_{i+1}K\neq \emptyset$ for all $i=0, \ldots, n-1$.
 Thus as $d_P(g_i, g_{i+1}) \leq D$, we have that $g$ and $h$ can be connected in $P$ outside $B(e, R)$.
 As $R$ was arbitrary, $P$ is one-ended. 
\end{proof}

\begin{samepage}
\subsection{Doubling}
\begin{definition}
	A metric space $(X,d)$ is \emph{$N$-doubling} if every ball can be covered by at most 
	$N$ balls of half the radius.
\end{definition}
\end{samepage}

Every hyperbolic group has doubling boundary, but this is not the case for relatively hyperbolic groups.
\begin{proposition}\label{prop-bdry-doubling}
	The boundary of a relatively hyperbolic group $(G, \cP)$ is doubling if and only
	if every peripheral subgroup is virtually nilpotent.
\end{proposition}
Recall that all relatively hyperbolic groups we consider are finitely generated, with $\cP$ a finite collection
of finitely generated peripheral groups.
\begin{proof}
	By \cite[Theorem 0.1]{DY-05-bdd-rel-hyp}, every peripheral subgroup is virtually nilpotent
	if and only if $X=X(G, \cP)$ has bounded growth at all scales: for every $0 < r < R$ there exists some $N$
	so that every radius $R$ ball in $X$ can be covered by $N$ balls of radius $r$.
	
	If $X$ has bounded growth at some scale then $\bdry X$ is doubling \cite[Theorem 9.2]{BS-00-gro-hyp-embed}.
	
	On the other hand, if $\bdry X$ is doubling, then $\bdry X$ quasi\-sym\-met\-rical\-ly embeds into some
	$\Sph^{n-1}$ (see \cite{Ass-83-snowflake}, or \cite[Theorem 12.1]{Hei-01-lect-analysis}).
	Therefore, $X$ quasi-isometrically embeds into some $\HH^n$ \cite[Theorems 7.4, 8.2]{BS-00-gro-hyp-embed},
	which has bounded growth at all scales.
	We conclude that $X$ has bounded growth at all scales (for small scales, the bounded growth of $X$ follows from
	the finiteness of $\cP$, and the finite generation of $G$ and all peripheral groups).
\end{proof}

\subsection{Partial self-similarity}

The boundary of a hyperbolic group $G$ with a visual metric $\rho$ is self-similar:
there exists a constant $L$ so that for any ball $B(z, r) \subset \bdry G$, with $r \leq \diam(\bdry G)$,
there is a $L$-bi-Lipschitz map $f$ from the rescaled ball $(B(z,r), \frac1r \rho)$ to an open set
$U \subset \bdry G$, so that $B(f(z), \frac1L) \subset U$.  
(See \cite[Proposition 3.3]{BK-11-coxeter} or \cite[Proposition 6.2]{BuLe-selfsim} for proofs that omit
the claim that $B(f(z), \frac1L) \subset U$.  The full statement follows from the lemma below.)

There is not the same self-similarity for the boundary $\bdry (G, \cP)$ of a relatively hyperbolic group $(G, \cP)$,
because $G$ does not act cocompactly on $X(G, \cP)$.
However, as we see in the following lemma, 
the action of $G$ does show that balls in $\bdry (G,\cP)$ with centres suitably far from parabolic points
are, after rescaling, bi-Lipschitz to large
balls in $\bdry (G, \cP)$.  The proof of Lemma~\ref{lem-self-sim} follows \cite[Proposition 3.3]{BK-11-coxeter} closely.

Partial self-similarity is essential in the following two sections, as we use it to control the geometry of the boundary
away from parabolic points.  Near parabolic points we use the asymptotic geometry of the corresponding peripheral 
group to control the geometry of the boundary.

\begin{lemma}\label{lem-self-sim}
	Suppose $X$ is a $\delta_X$-hyperbolic, proper, geodesic metric space with base point $w \in X$.
	Let $\rho$ be a visual metric on the boundary $\bdry X$ with parameters $C_0, \eps$.
	Then for each $D > 0$ there exists $L_0=L_0(\delta_X, \eps, C_0, D) < \infty$ with the following property:
	
	Whenever we have $z \in \bdry X$ and an isometry $g \in \Isom(X)$ so that some $y \in [w,z)$ satisfies
	$d(g^{-1} w, y) \leq D$,
	then $g$ induces an $L_0$-bi-Lipschitz map $f$ 
	from the rescaled ball 
	$( B(z, r), \frac1{r}\rho )$, where $r=e^{-\eps(d(w,y)+\del_X+1)}/2C_0$,
	to an open set $U \subset \bdry G$, so that $B(f(z), \frac1{L_0}) \subset U$.
\end{lemma}
\begin{proof}
	We assume that $z$, $g$ and $r$ are fixed as above.
	We use the following equality:
	\begin{equation}\label{eq-bk-claim2} 
		-\frac1\eps \log(2r C_0)=d(w,y)+\delta_X+1.
	\end{equation}
	
	For every $z_1, z_2 \in B(z, r)$, and every $p \in (z_1, z_2)$,
	one has 
	\begin{equation}\label{eq-bk-claim}
		d(y, [w, p]) \leq 3\delta_X.
	\end{equation}
	This is easy to see: $\rho(z_1, z_2) \leq 2r$, so 
	$d(w, (z_1, z_2)) \geq -\frac1\eps \log(2r C_0)$.
	Let $y_1 \in [w, z_1)$ be so that $d(y_1, w) = d(y, w)$, and notice that
	$d(y_1, (z_1, z_2)) > \delta_X$ by \eqref{eq-bk-claim2}.
	For any $p \in (z_1, z_2)$, the thinness of the 
	geodesic triangle $w, z_1, p$ implies that
	$d(y_1, [w,p]) \leq \delta_X$.
	In particular, for $z_2 = p = z$, we have
	$d(y_1, [w, z)) \leq \delta_X$,
	so $d(y_1, y) \leq 2\delta_X$,
	and the general case follows.
	
	Let $g \in G$ be given so that $d(g^{-1}w, y) \leq D$.
	For any $z_1, z_2 \in B(z, r)$, by \eqref{eq-bk-claim2}, \eqref{eq-bk-claim}
	we have
	\begin{align*}
		d(w, (z_1, z_2)) & \approx_{6\del_X} d(w, y)+d(y, (z_1, z_2)) \\
			& \approx_{(7\del_X+D+1)}  \frac{-1}{\eps}\log(2r C_0) + d(g^{-1}w, (z_1, z_2)).
	\end{align*}
	As $d(g^{-1}w, (z_1, z_2)) = d(w, (gz_1, gz_2))$, this gives that
	\[
		L_0^{-1} \frac{\rho(z_1, z_2)}{r} \leq \rho(g z_1, g z_2) \leq L_0 \frac{\rho(z_1, z_2)}{r},
	\]
	for any choice of $L_0 \geq 2C_0^3 e^{\eps(7\delta_X+D+1)}$.
	
	Thus the action of $g$ on $B(z, r)$ defines a $L_0$-bi-Lipschitz map $f$ with image $U$,
	which is open because $g$ is acting by a homeomorphism.
	It remains to check that $B(f(z), \frac{1}{L_0}) \subset U$.
	
	Suppose that $z_3 \in B(f(z), \frac1{L_0})$.
	Then $d(w, (f(z),z_3)) > \frac{-1}{\eps} \log(C_0/L_0)$, but
	$d(w, (f(z),z_3)) = d(g^{-1}w, (z, g^{-1}z_3))$. So, for large enough $L_0$, we have
	\begin{align*}
		d(w, (z, g^{-1}z_3)) & \geq \frac{-1}{\eps} \log\left(\frac{C_0}{L_0}\right)
		+ d(w, y)-C_1(\delta_X, D) \\
			& > \frac{-1}{\eps}\log\left(\frac{C_0}{L_0}\right) -\frac1\eps \log(2r C_0)-C_2(C_1, \delta_X) \\
			& > \frac{-1}{\eps} \log\left(\frac{r}{C_0}\right),
	\end{align*}
	where the last equality follows from increasing $L_0$ by an amount depending only on $\eps, C_0, C_2$.
	We conclude that $\rho( z, g^{-1}z_3) < r$.
\end{proof}

In our applications, it is useful to reformulate Lemma~\ref{lem-self-sim} so the input of the property is a ball in $\bdry X$ 
rather than an isometry of $X$.
\begin{corollary}[Partial self-similarity]\label{cor-self-sim2}
	Let $X$, $\bdry X$, $\rho$, $C_0$ and $\eps$ be as in Lemma~\ref{lem-self-sim}.
	Suppose $G$ acts isometrically on $X$.
	Then for each $D>0$ there exists $L_0 = L_0(\del_X, \eps, C_0, D) < \infty$ with the following property:
	
	Let $z \in \bdry X$ and $r \leq \diam(\bdry X)$ be given, and set 
	\[ d_r = - \frac1\eps \log(2rC_0) - \delta_X-1. \]  Then	
	\begin{enumerate}
	\item
	If $d_r \geq 0$, set $x \in [w,z)$ so that $d(w,x) = d_r$.
	Then for any $y \in [w, x]$ so that $d(g^{-1}w, y) \leq D$, for some $g \in G$,
	there exists a $L_0$-bi-Lipschitz map $f$ (induced by the action of $g$ on $\bdry X$)
	from the rescaled ball 
	$( B(z, r'), \frac1{r'}\rho )$, where $r' = r e^{\eps d(x,y)}$,
	to an open set $U \subset \bdry G$, so that $B(f(z), \frac1{L_0}) \subset U$.
	
	\item	
	If $d_r < 0$, then the identity map on $\bdry X$ defines a $L_0$-bi-Lipschitz map from the rescaled ball
	$( B(z, r'), \frac1{r'}\rho )$, where $r' = r$,
	to an open set $U \subset \bdry G$, so that $B(f(z), \frac1{L_0}) \subset U$.
	\end{enumerate}
\end{corollary}
\begin{proof}
	Let $L_0'$ be the value of $L_0$ given by Lemma \ref{lem-self-sim}.
	Since $d(w,y) = d_r-d(x,y)$, and
	$e^{-\eps(d_r-d(x,y)+\del_X+1)}/2C_0 = r'$, part (1) follows from Lemma~\ref{lem-self-sim}.
	
	Note that if $d_r < 0$, then $r > 1/C_1 >0$ for some $C_1=C_1(\eps, C_0, \del_X) < \infty$, so 
	part (2) follows from setting $L_0 = \max\{L_0', C_1\}$.
\end{proof}

\section{Linear Connectedness}\label{sec-lin-conn}

Under the hypotheses of Theorem~\ref{thm-bowditch}, we saw that $\bdry(G, \cP)$
is connected and locally connected.  In this section we show that $\bdry(G, \cP)$
 satisfies a quantitatively controlled version of this property.

\begin{definition}
	We say a complete metric space $(X,d)$ is \emph{$L$-linearly connected} for some $L \geq 1$ if
	for all $x,\,y \in X$ there exists a compact, connected set $J \ni x,\,y$
	of diameter less than or equal to $L d(x,y)$.
\end{definition}
This is also called the $L$-bounded turning property in the literature.
Up to slightly increasing $L$, we can assume that $J$ is an arc, see \cite[Page 3975]{Mac-08-quasi-arc}.

\begin{proposition}\label{prop-bdry-lin-conn}
	If $(G, \cP)$ is relatively hyperbolic and $\bdry(G, \cP)$ is connected and locally connected with no global cut points,
	then $\bdry(G, \cP)$ is linearly connected.
\end{proposition}

If $\cP$ is empty then $G$ is hyperbolic, and this case is already known by work of 
Bonk and Kleiner \cite[Proposition 4]{BK-05-quasiarc-planes}.
Lemma \ref{lem-self-sim} gives an alternate proof of this result, which we include to warm up for the proof of
Proposition \ref{prop-bdry-lin-conn}.
Both proofs rely on the work of Bestvina and Mess, and Bowditch and Swarup cited in the introduction.
\begin{corollary}[Bonk-Kleiner]\label{cor-hyp-lin-conn}
	If the boundary of a hyperbolic group $G$ is connected, then it is linearly connected.
\end{corollary}
\begin{proof}
	Let $X=\Gamma(G)$ by a Cayley graph of $G$ with visual metric $\rho$,
	and let $L_0$ be chosen by Corollary~\ref{cor-self-sim2} for $D=1$.
	The boundary of $G$ is locally connected 
	\cite{BM-91-dimbdry,Bow-98-bdry-access-hyp,Bow-99-conn-lim-set,Swa-96-cut-point},
	so for every $z \in \bdry G$, we can find an open connected set $V_z$ satisfying
	$z \in V_z \subset B(z, 1/{2L_0})$.
	The collection of all $V_z$ forms an open cover for the compact space $\bdry G$,
	and so this cover has a Lebesgue number $\alpha > 0$.
	
	Suppose we have $z, z' \in \bdry G$.
	Let $r = \frac{2L_0}\alpha \rho(z,z')$.
	If $r > \diam(\bdry X)$, we can join $z$ and $z'$ by a set of diameter
	$\diam(\bdry X) < \frac{2L_0}{\alpha}\rho(z,z')$.
	
	Otherwise, we apply Corollary~\ref{cor-self-sim2}, using either (1) with $y=x$ or (2), 
	to find an $L_0$-bi-Lipschitz
	map $f: (B(z,r),\frac{1}{r}\rho) \ra U$.
	Since $\rho(f(z), f(z')) \leq L_0 \rho(z,z')/r = \frac\alpha{2}$, we
	can find a connected set $J \subset B(f(z),\frac1{L_0}) \subset U$
	that joins $f(z)$ to $f(z')$.
	Therefore $f^{-1}(J) \subset B(z,r)$ joins $z$ to $z'$, and has diameter at most
	$2r = \frac{4L_0}{\alpha}\rho(z,z')$.
	So $\bdry G$ is $4L_0/\alpha$-linearly connected.
\end{proof}

The key step in the proof of Proposition \ref{prop-bdry-lin-conn} is the construction of chains of points in the boundary.
\begin{lemma}\label{lem-bdry-chains}
	Suppose $(G, \cP)$ is as in Proposition~\ref{prop-bdry-lin-conn}.
	Then there exists $K$ so that
	for each pair of points $a,b\in \bdry(G,\cP)$ there exists a chain of points $a=c_0,\dots, c_n=b$ such that
	\begin{enumerate}
		\item for each $i=0,\dots,n$ we have $\rho(c_i,c_{i+1})\leq \rho(a,b)/2$, and
		\item $\diam(\{c_0,\dots,c_{n}\})\leq K \rho(a,b)$.
	\end{enumerate}
\end{lemma}
We defer the proof of this lemma.
\begin{proof}[Proof of Proposition \ref{prop-bdry-lin-conn}]
Given $a, b \in \bdry(G,\cP)$, apply Lemma~\ref{lem-bdry-chains} to get
a chain of points $J_1=\{c_0,\dots,c_n\}$.
For $j \geq 1$, we define $J_{j+1}$ iteratively by applying Lemma~\ref{lem-bdry-chains} to
each pair of consecutive points in $J_j$, and concatenating these chains of points together.
Notice that
\[
	\diam(J_{j+1})\leq \diam(J_j)+\frac{2 K}{2^j} \rho(a,b).
\]
This implies that the diameter of $J=\overline{\bigcup J_j}$ is linearly bounded in $\rho(a,b)$,
and $J$ is clearly compact and connected as desired.
\end{proof}

We require two further lemmas before commencing the proof of Lemma \ref{lem-bdry-chains}.
The first is an elementary lemma on the geometry of infinite groups.

\begin{lemma}\label{ray}
 Let $P$ be an infinite, finitely generated group with Cayley graph $(\Gamma(P), d_P)$.
 Then for each 
 $p,q\in P$ there exists a geodesic ray $\alpha$ starting from $p$ and such that $d_P(q,\alpha)\geq d_P(p,q)/3$.
\end{lemma}

\begin{proof}
 As $P$ is infinite, there exists a geodesic line $\gam$ through $p$, which can be subdivided into geodesic 
 rays $\alpha_1, \alpha_2$ starting from $p$. We claim that either $\alpha_1$ or $\alpha_2$ satisfies the requirement.
 In fact, if that was not the case we would have points $p_i\in\alpha_i\cap B(q,d_P(p,q)/3)$. 
 Notice that $d_P(p_i,p)\geq 2d_P(p,q)/3$. Now, 
$$d_P(p_1,p_2)\leq d_P(p_1,q)+d_P(q,p_2)\leq 2d_P(p,q)/3,$$
but this contradicts
\[ d_P(p_1,p_2)=d_P(p_1,p)+d_P(p,p_2)\geq 4d_P(p,q)/3. \qedhere \]
\end{proof}

The next lemma describes the geometry of geodesic rays passing through a horoball.
If $a, b \in \bdry X$, we use the notation $p_{a,b}$ for the centre of the quasi-tripod $w, a, b$, 
i.e.\ the point in $[w, a) \subset X$ such that $d(w,p_{a,b})=(a |b)$. 

\begin{lemma}\label{prodhorrel}
 Fix $O\in\cO$ and $a,b\in\bdry(G,\cP)\backslash \{a_O\}$. Let $\gam_a,\gam_b$ be geodesics 
 from $a_O$ to $a,b$ and let $q_a,q_b$ be the last points
 in $\gam_a\cap \overline{O},\gam_b\cap \overline{O}$, which we assume to be both non-empty. 
 Also, let $\gam$ be a geodesic from $w$ to $a_O$ and let $q$ be the first point in $\gam\cap \overline{O}$
 (so that $d_O \approx d(w,q)$).
 Then there exists $E = E(X)< \infty$ so that the following holds.
\begin{enumerate}
 \item If $(a|a_O)\geq d_O$ then
$$(a|a_O)\approx_E d(q_a,q)/2+d_O.$$
 \item If $(a|a_O),(b|a_O) \in [d_O, (a|b)]$ then
$$(a|b)\gtrsim_E 2(a|a_O)-d_O-d(q_a,q_b)/2 \approx_E d(q_a,q)+d_O-d(q_a,q_b)/2.$$
Moreover, if $d(q_a,q_b)\geq E$ then $\approx_E$ holds in the equation above.
 \item If $(a|a_O)< d_O$ then $d(q_a,q)\approx_E 0$.
 \item If $p_{a,b} \in O$ and $d(p_{a,b}, X\setminus O) \geq R \geq E$ then $d(p_{a, a_O}, X \setminus O)$ and
 $d(p_{b, a_O}, X \setminus O)$ are both at least $R-E$, and $d(q_a, q_b) \geq 2R-E$.
\end{enumerate}

\end{lemma}

\begin{proof}
As in Lemma \ref{farfrombpoints} we only need to make the computations in the case of trees, illustrated by Figure~\ref{fig-prodhorrel}, and an approximation argument gives in each case the desired inequalities. 

\begin{figure}[h]
 \includegraphics[scale=0.45]{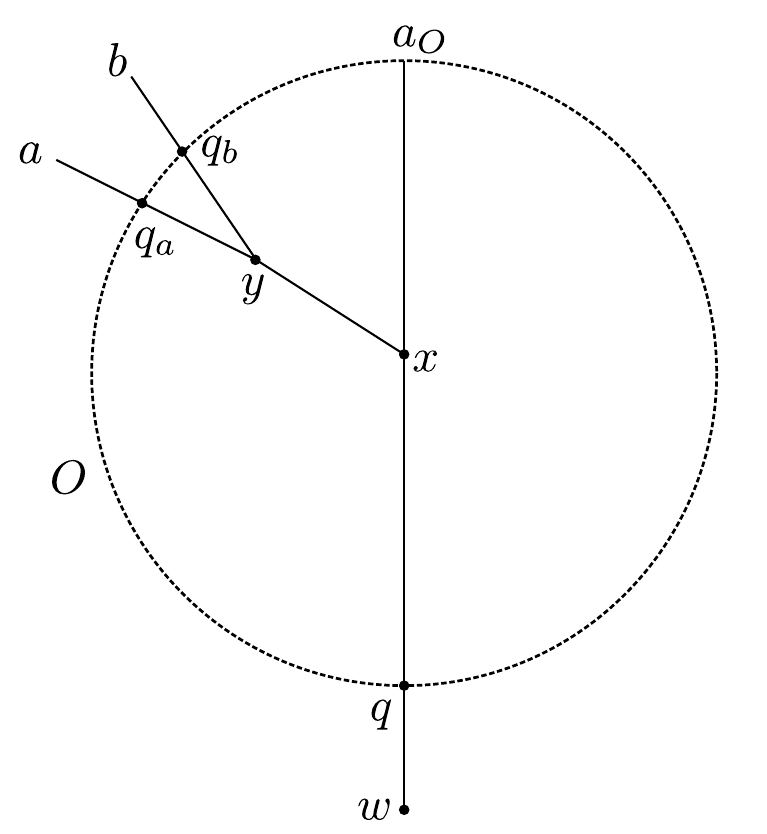}
 \caption{Geodesics passing through a horoball}
 \label{fig-prodhorrel}
\end{figure}

\par
$(1)$ Keeping into account that $x$ lies in $O$ as $d(w,x)=(a|a_O)\geq d_O$, the computation in a tree yields
$$(a|a_O)=d(w,q)+d(q,x)=d_O+d(q_a,q)/2,$$
as $d(x,q)=d(x,q_a).$
\par
$(2)$ The figure illustrates the first of the two possible types of tree approximating the configuration 
we are interested in. The second case to consider is when $q_a, q_b$ are between $x$ and $y$, and thus $q_a = q_b$ in the tree.
Therefore, for a suitable choice of $E$, the ``moreover'' assumption ensures we are in the first case.
In this first case we have the equality:
$$(a|b)=d(w,x)+d(q,x)-d(q_a,q_b)/2=(a|a_O)+((a|a_O)-d_O)-d(q_a,q_b)/2.$$

In the second case we can proceed similarly. We see that 
\begin{align*}
	(a|b) & \geq d(w, q_a) = d(w,x)+d(x,q_a) 
	\\ & = (a|a_0) + ((a|a_O)-d_O)
	= 2(a|a_O) - d_O,
\end{align*}
which is what we need as $d(q_a, q_b) =0$.
In both cases the final $\approx$ follows from part (1).
\par
$(3)$ In the tree approximating this configuration the ray from $w$ to $a$ does not enter the horoball $O$, so that the bi-infinite geodesic $\gamma_a$ exits $O$ from $q$.
\par
$(4)$ There are two types of tree approximating this configuration.  The first is given by Figure~\ref{fig-prodhorrel},
where $p_{a,a_O} = p_{b,a_O} = x$ and $p_{a,b} = y$, so 
\[
	d(p_{a,a_O}, X \setminus O) = d(p_{b, a_O}, X \setminus O) = d(x, q) = d(x, q_a) \geq d(y, q_a) \geq R.
\]
In this case $d(q_a, q_b) = 2 d(p_{a,b}, q_a) \geq 2R$.

The second configuration is when the geodesics $[w, a)$ and $[w, b)$ branch off from $[w, a_O)$ at different points.
Suppose that $(a|a_O) < (b|b_O)$, so $p_{a, a_O}$ lies on $[w, a_O)$ strictly between $q$ and $p_{b, b_O}$, and thus
$p_{a,b} = p_{a, a_O}$.  In this case $d(q_a, q_b) \geq d(q_a, p_{a,b}) + d(p_{a,b}, q_b) \geq 2R$, and we have
\[
	d(p_{b, a_O}, X \setminus O) = d(p_{b, a_O}, q) > d(p_{a,b}, q) \geq R. \qedhere
\]
 
\end{proof}

For each peripheral subgroup $P \in \cP$ we denote by $d_P$ the path metric on any left coset of $P$.

We are now ready to commence the proof of Lemma \ref{lem-bdry-chains}.
This proof is somewhat delicate, splitting into two cases, depending on the position of the 
the points $a, b \in \bdry(G, \cP)$.
In the first case, we use Lemma~\ref{ray} and the asymptotic geometry of a horoball to join $a$ and $b$ by a chain of points.
The second case is similar to Corollary~\ref{cor-hyp-lin-conn} for hyperbolic groups: the partial self-similarity of 
the boundary upgrades local connectedness to linear connectedness for $a$ and $b$.
A final argument in Case 2b uses the group action and the no global cut point condition to cover the remaining configurations.

\begin{proof}[Proof of Lemma \ref{lem-bdry-chains}.]
We need to find chains of points joining distinct points $a,b \in \bdry (G,\cP)$,
as described in the statement of the lemma.

Recall that $p_{a,b} \in [w, a)$ denotes the point on $[w, a) \subset X=X(G,\cP)$ 
such that $d(w,p_{a,b})=(a |b)$. 

Let $R=R(X)$ be a large constant to be determined by Case 1 below.
All constants may depend tacitly on $C_0, \eps, \del_X$.

\vspace{1mm}
{\noindent \textbf{Case 1:}}
We first assume that there exists $O\in\cO$ such that $p_{a,b}\in O$ and $d(p_{a,b},gP)>R$, 
where $gP$ is the left coset of the peripheral subgroup $P$ corresponding to $O$.

\vspace{1mm}
{\noindent \textbf{Case 1a:}}
Suppose that $\rho(a,b)\leq \rho(a,a_O)/S$ and $\rho(a,b) \leq \rho(b,a_O)/S$, 
for some large enough $S>1$ to be determined.
In this case, we push an appropriate geodesic path in $gP$ out to the boundary.

Let $S' = \log(S/C_0^2)/\eps$, and note that $(a|b)-(a|a_O) \geq S'$ and $(a|b)-(b|a_O) \geq S'$.  We assume that $S' \geq 0$.

Let $\gam_a$ be a geodesic from $a_O$ to $a$ and let 
$q_a$ be the last point in $\gam_a\cap \overline{O}$. Define $\gam_b$ and $q_b$ analogously. 
Let $\gam$ be a geodesic from $w$ to $a_O$ and let $q$ be the first point in $\gam \cap \overline{O}$.

Assuming $R \geq E$, by Lemma~\ref{prodhorrel}(4)
we have \[(a|a_O) = d(w, p_{a,a_O}) \geq d_O + d(p_{a,a_O}, X \setminus O) \geq d_O + R-E \geq d_O,\]
and likewise $(b|a_O) \geq d_O$.
Using Lemma~\ref{prodhorrel}(1) and the approximate equality case of Lemma~\ref{prodhorrel}(2), we have
\begin{equation}\label{eq-chain5}
\begin{split}
	d(q_a,q) & \approx_{2E} 2((a|a_O)-d_O) \\
		& \geq 2((a|a_O)-d_O)) +2(S'-(a|b)+(a|a_O)) \\
		& =2(2(a|a_O)-(a|b)-d_O+S')\approx_{2E} d(q_a,q_b)+2S'.
\end{split}
\end{equation}

We now define our chain of points joining $a$ to $b$.
Let $\alpha$ be a geodesic in $gP$ connecting $q_a$ to $q_b$, and denote by
$q_a=q_0,\dots,q_n = q_b$ the points of $\alpha \cap gP$.
For $i=0, \ldots, n-1$ let $c_i = q_i q_0^{-1} a \in\bdry (G,\cP)$ be the
endpoint of $q_i q_0^{-1}\gam_a$ other than $a_O$, and set $c_n = b$.
Notice that \[ 2\log(d_P(q,q_a)/d_P(q_a,q_b)) \approx_{2A} d(q,q_a)-d(q_a,q_b) \gtrsim_{4E} 2S' \] 
by Lemma~\ref{lem-horoball-dist2} and \eqref{eq-chain5}, so
for $S = S(E,A)$ large enough,
\[ 2 \log(d_P(q,q_a)/d_P(q_a,q_b)-1) \geq S',\] thus
\begin{align*}
	d(q,q_i)& \approx_A 2\log(d_P(q,q_i))\geq 2\log(d_P(q,q_a)-d_P(q_a,q_i)) \\
		& \geq 2\log(d_P(q,q_a)-d_P(q_a,q_b)) \\
		& \geq S' + 2 \log(d_P(q_a, q_b)) \approx_{A} S'+ d(q_a,q_b).
\end{align*}
In particular, if $S = S(E,A)$ is large enough we have $(c_i|a_O)\geq d_O$ for each $i$, by Lemma~\ref{prodhorrel}(3).
By Lemma~\ref{lem-horoball-dist2}, as $d_P(q_a, q_i) \leq d_P(q_a, q_b)$ we have 
$d(q_a, q_i) \lesssim_{2A} d(q_a, q_b)$.  Thus Lemma~\ref{prodhorrel}(2) gives 
\begin{align*}
	(a|c_i) & \gtrsim_{E} 2(a|a_O)-d_O-d(q_a,q_i)/2 \\ & \gtrsim_{A} 2(a|a_O)-d_O-d(q_a,q_b)/2 \ \ \approx_E (a|b),
\end{align*}
which gives the distance bound $\rho(a,c_i) \leq C_1 \rho(a,b)$, for $C_1=C_1(E,A)$.
This gives the diameter bound $\diam(\{c_i\}) \leq K_1 \rho(a,b)$ for $K_1 = 2C_1$.

We saw that $(a|c_i) \gtrsim_{2E+A} (a|b) \geq S' +(a|a_O)$, and so
$(c_i|a_O) \gtrsim \min\{(c_i|a), (a|a_O)\} \gtrsim (a|a_O)$, with error $C_2 = C_2(E, A)$.
We also have 
$(c_i|c_{i+1}) \gtrsim_1 d(q,q_i)+d_O \approx_E (c_i|a_O)+\frac12 d(q,q_i)
\gtrsim_{A} (c_i|a_O) + \frac{1}{2}S'$, so for $S=S(E,A)$ large enough, $(c_i|c_{i+1}) \geq (c_i|a_O)$ and likewise $(c_i|c_{i+1}) \geq (c_{i+1}|a_O)$.
Applying Lemma~\ref{prodhorrel}(2) twice and Lemma~\ref{prodhorrel}(4) we see that
\begin{align*}
	(c_i|c_{i+1}) & \gtrsim_{E} 2(c_i|a_O) - d_O - d(q_i, q_{i+1})/2
		\gtrsim_{2C_2+1} 2(a|a_O) - d_O \\ & \approx_E (a|b)+d(q_a, q_b)/2 \gtrsim_E (a|b)+R,
\end{align*}
and so for $R \geq R_1(C_2, E)$ we have $\rho(c_i,c_{i+1}) \leq \rho(a,b)/2$.

\vspace{1mm}
{\noindent \textbf{Case 1b:}}
Suppose that $b=a_O$.  In this case, a chain of points joining $a$ and $a_O$ is found by using an appropriate
geodesic ray in $gP$ and pushing it out to the boundary.  For a suitable choice of $R$, 
depending on the value of $S$ fixed by Case 1a, we will actually
ensure that the distance between subsequent points in the chain is at most $\rho(a, a_O)/2(S+1)$.

Let $\gam_a$, $q_a$, $\gam$ and $q$ be as above.
Notice that $q,q_a$ lie on $gP$,
so by Lemma~\ref{prodhorrel}(1)
\begin{equation}\label{eq-chain1}
	(a|a_O) \approx_{E} d(q_a,q)/2 + d_O \geq R+d_O.
\end{equation}

By Lemma~\ref{ray}, there exists a geodesic ray $\alpha$ in $gP$ starting at $q_a$ such that 
$d_P(q,\alpha)\geq d_P(q_a,q)/3$.
Therefore, by Lemma~\ref{lem-horoball-dist2} and \eqref{eq-chain1},
\begin{equation}\label{eq-chain2}
	d(q, \alpha) \approx_A 2 \log(d_P(q, \alpha))\geq 2\log(d_P(q_a,q)/3)\approx_{A+3} d(q_a,q) \geq 2R.
\end{equation}
Let $q_a=q_0,\dots,q_n,\dots$ be the points of $\alpha \cap gP$,
and, as before, for each $i \geq 0$ let $c_i = q_i q_0^{-1} \gam_a \in\bdry (G,\cP)$.

By Lemma~\ref{prodhorrel}(3) and \eqref{eq-chain2}, for $R \geq R_2(E,A) \geq R_1$, 
we can assume that $(c_i|a_O)\geq d_O$ for each $i$.

Using Lemma~\ref{prodhorrel}(1) and \eqref{eq-chain2}, there exists $C_3=C_3(A)$ so that
\begin{equation*}
	(c_i|a_O)\approx_{E} d(q_{i},q)/2+d_O \gtrsim_{C_3} d(q_a,q)/2 +d_O\approx_{E} (a|a_O).
\end{equation*}
And consequently there exists $C_4=C_4(C_3, E)$ so that for each $i$
\begin{equation}\label{eq-chain3}
	\rho(c_i,a_O)\leq C_4 \rho(a,a_O);
\end{equation}
this gives $\diam(\{c_i\}) \leq K_2 \rho(a, a_O)$ for $K_2 = 2C_4$.

Similarly to Case 1a, we have $(c_i|c_{i+1}) \gtrsim_{1+E} (c_i|a_O)+\frac12 d(q,q_i)$ and $d(q,q_i) \gtrsim_{2A+3} 2R$, so for $R \geq R_3(E,A) \geq R_2$ we have $(c_i|c_{i+1}) \geq \max\{(c_i|a_O),(c_{i+1}|a_O)\}$.
By Lemma~\ref{prodhorrel}(2), \eqref{eq-chain1} and \eqref{eq-chain2}, we have for $C_5=C_5(A)$
\begin{align*}
	(c_i|c_{i+1}) & \gtrsim_{E} d(q_i,q)+d_O-d(q_{i},q_{{i+1}})/2 
		 \gtrsim_{C_5} d(q_a,q)+d_O \\ & = \big(d(q_a,q)/2+d_O\big) + d(q_a,q)/2
		 \gtrsim_{E} (a|a_O)+R.
\end{align*}
So, taking $R \geq R_4(C_5, E, S) \geq R_3$, we have
\begin{equation}\label{eq-chain4}
	\rho(c_i,c_{i+1})\leq \rho(a,a_O)/2(S+1).
\end{equation}

For each $i$, by Lemmas~\ref{lem-horoball-dist2} and \ref{prodhorrel}(1), 
$(c_i|a_O)\approx_{E+A} \log(d_P(q_i,q))+d_O$,
so for $N$ large enough
we have  $\rho(c_N,a)\leq \rho(a,a_O)/2(S+1)$. 

Therefore the chain of points $a=c_0,\dots, c_N, a_O=b$ satisfies our requirements
by \eqref{eq-chain3}, \eqref{eq-chain4}.

\vspace{1mm}
{\noindent \textbf{Case 1c:}}
In this case, we have $\rho(a,b)\geq \rho(a,a_O)/S$ or $\rho(a,b)\geq \rho(b,a_O)/S$.
Without loss of generality, we assume that $\rho(a, a_O) \leq \rho(b, a_O)$ and $\rho(a, a_O) \leq S \rho(a,b)$.

Assume that $R \geq R_5 = R_4+E$.
Then by Lemma~\ref{prodhorrel}(4), $d(p_{a, a_O}, X \setminus O)$ and $d(p_{b, a_O}, X \setminus O)$ are 
both at least $R_4$, so by Case 1b there exist 
chains $a=c_0, c_1, \ldots, c_m=a_O$ and $a_O=c_0', c_1', \ldots c_n'=b$, with, for each $i$,
\begin{gather*}
	\rho(c_i, c_{i+1})  \leq \frac{\rho(a, a_O)}{2(S+1)} \leq \frac{S \rho(a,b)}{2(S+1)} < \frac{\rho(a,b)}{2}, \text{ and}\\
	\rho(c_i', c_{i+1}')  \leq \frac{\rho(b, a_O)}{2(S+1)} \leq \frac{\rho(b,a) + \rho(a,a_O)}{2(S+1)} \leq \frac{\rho(a,b)}{2}.
\end{gather*}
The diameter of $\{c_i\} \cup \{c_i'\}$ is at most $K_2 \rho(a, a_O) + K_2 \rho(b, a_O) \leq K_3 \rho(a,b)$ for
$K_3 = (2S+1)K_2$.

\vspace{1mm}
{\noindent \textbf{Case 2:}}
We assume that $d(p_{a,b}, \Gamma(G)) < R$.  In this case we can use the group action to find a connected set
joining $a$ and $b$ directly.

	Let $L_0 > 1$ be given by Corollary~\ref{cor-self-sim2} applied to $X$ with $D=R$.
	Since $\bdry(G, \cP)$ is locally connected and compact,
	there exists $\alpha > 0$ so that any $B(z,\alpha) \subset \bdry(G, \cP)$ is contained in an open, connected set of diameter
	less than $1/L_0$ (see the proof of Corollary~\ref{cor-hyp-lin-conn}).
	
	Let $r_1 = 2 \frac{L_0}\alpha \rho(a,b)$ and let $y_1 \in [w, a)$ be chosen
	so that $d(w,y_1) = -\frac1\eps \log(2r_1C_0)-\delta_X-1$.
	If no such $y_1$ exists, then we are done as $\rho(a,b) \geq C_6 = C_6(L_0/\alpha) > 0$, so
	we can join $a$ and $b$ by a connected set of diameter $\leq K_4 \rho(a,b)$, for
	$K_4 = \diam(\bdry X)/C_6$.
	
	Let $t \gg 0$ be a large constant to be determined by Case 2b.

\vspace{1mm}
{\noindent \textbf{Case 2a:}}	If there exists $y \in [w, y_1]$ so that $d(y_1, y) \leq 3t$ and
	$d(y, G w) \leq D$, then we argue as in the proof of Corollary~\ref{cor-hyp-lin-conn}.
	
	By Corollary~\ref{cor-self-sim2}(1), using $z=a, r=r_1, x=y_1$ and $y$ as given,
	there exists an $L_0$-bi-Lipschitz map $f: (B(a,r'), \frac{1}{r'}\rho) \ra U$,
	where $r' = r_1 e^{\eps d(y_1, y')}$, so that $B(f(a), 1/L_0) \subset U$.
	
	Now,
	\[
		\rho(f(a), f(b)) \leq L_0 \cdot \frac{1}{r'} \rho(a,b) \leq \frac{L_0}{r_1} \rho(a,b) = \frac{\alpha}{2},
	\]
	so we can join $f(a)$ and $f(b)$ by a connected set $J \subset B(f(a), 1/L_0)$.
	Therefore we can join $a$ and $b$ by $f^{-1}(J) \subset B(a, r')$.
	As $r' \leq r_1 e^{\eps 3t}$, $f^{-1}(J)$ has diameter at most $2r' \leq K_5 \rho(a,b)$, for $K_5 = 4L_0e^{\eps 3t}/\alpha$.
	
	\begin{figure}
	\includegraphics[scale=0.45]{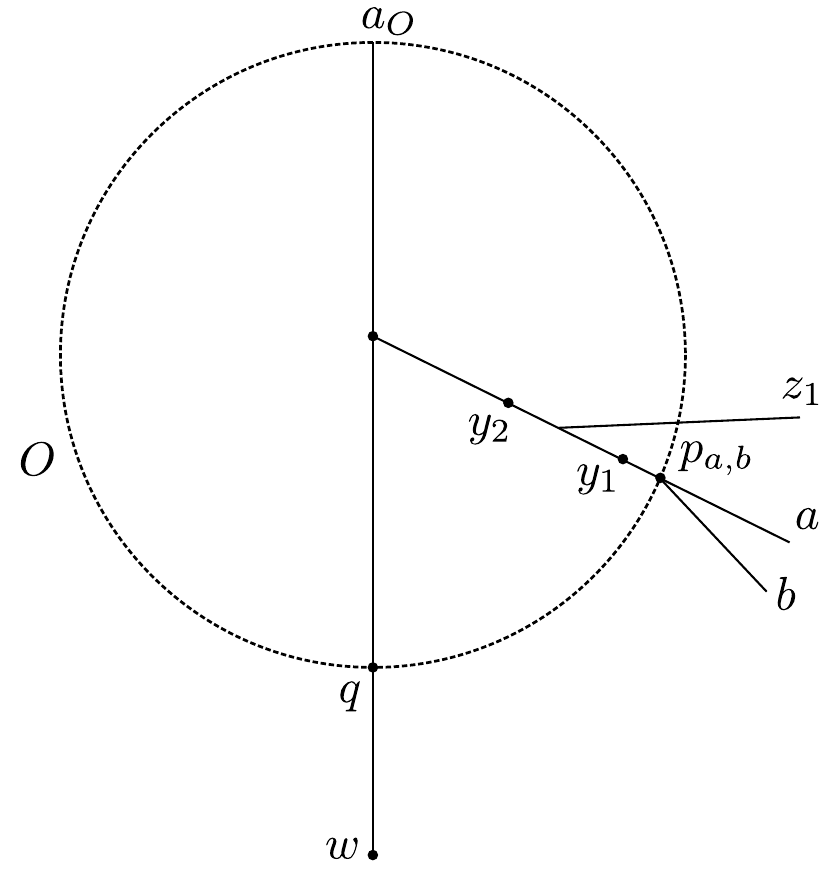}
	\caption{Lemma~\ref{lem-bdry-chains}, Case 2b}\label{fig-parabolicrescale}
	\end{figure}
\vspace{1mm}
{\noindent \textbf{Case 2b:}}
	If no such $y$ exists, we are in the situation of Figure~\ref{fig-parabolicrescale}.
	In this case, we use the absence of global cut points to find a connected set between $a$ and $b$.
	
	Let $y_2 \in [w, y_1)$ be chosen so that $d(y_1,y_2) = t$ and let $O \in \cO$ be the horoball containing $[y_2, y_1]$, which
	corresponds to the coset $gP$.
	Let $d_O = d(w, O)$, and let $p_1 \in gP$ be chosen so that $d(p_1, p_{a,b}) < R$.
	(In the figure, $p_1 = p_{a,b}$.)
	
	Let $\rho_1$ be a visual metric on $\bdry X = \bdry(G, \cP)$ based at $p_1$.
	We can assume that $(\bdry X, \rho)$ and $(\bdry X, \rho_1)$ are isometric, with the
	isometry induced by the action of $p_1$.
	In the metric $\rho_1$, we have that $a$, $b$ and $a_O$ are points separated by at least $\delta_0 = \delta_0(R)$.
	
	The boundary $(\bdry X, \rho)=(\bdry X, \rho_1)$ is compact, locally connected and connected.
	Consequently, given a point $c \in \bdry X $ that is not a global cut point, and $\delta_0 > 0$,
	there exists $\delta_1=\delta_1(\delta_0, c, \bdry X) > 0$ so that any two points in $\bdry X \setminus B(c,\delta_0)$
	can be joined by an arc in $\bdry X \setminus B(c, \delta_1)$.
	
	In our situation, 
	$\delta_1$ may be chosen to be independent of the choice of (finitely many) $c = a_O$ satisfying $d(O, p_1) = 0$,
	so $\delta_1 = \delta_1(\delta_0, \bdry X)$.
	Therefore, $a$ and $b$ can be joined by a compact arc $J$ in $(\bdry X, \rho_1)$ that does
	not enter $B_{\rho_1}(a_O, \delta_1)$.
	So geodesic rays from $p_1$ to points in $J$ are at least $2\delta_X$ from the geodesic ray
	$[p_{a,b}, a_O)$ outside the ball $B(p_1, t)$, for $t = -\frac1\eps \log(\delta_1)+C_7$, 
	where $C_7 = C_7(C_0, \eps, \del_X, R)$.
	
	Translating this back into a statement about $(\bdry X, \rho)$, we see that 
	geodesics from $w$ to points in $J$ must branch from $[w, p_{a,b}]$ after $y_2$, that is,
	the set $J$ lies in the ball $B(a, (K_6/2)\rho(a,b))$, for $K_6 = K_6(d(p_{a,b}, y_2)) = K_6(r_1, t)$.
	
	From these connected sets of controlled diameter, it is easy to extract chains of points
	satisfying the conditions of the lemma, with $K = \max\{K_1, \ldots, K_6\}$.
\end{proof}

\section{Avoidable sets in the boundary}\label{sec-avoiding}

In order to build a hyperbolic plane that avoids horoballs, we need to build an arc
in the boundary that avoids parabolic points.
In Theorem~\ref{thm-main-hyp-rel-hyp}, we also wish to avoid the specified hyperbolic subgroups.
We have topological conditions such as the no local cut points condition which help,
but in this section we find more quantitative control.

Given $p \in X$, and $0 < r< R$, the \emph{annulus} $A(p,r,R)$ is defined to be
$\overline{B}(p,R) \setminus B(p,r)$.  More generally, we have the following.
\begin{definition}\label{def-annular-nbhd}
	Given a set $V$ in a metric space $Z$, and constants $0 < r < R < \infty$,
	we define the \emph{annular neighbourhood}
	\[
		A(V,r,R) = \{ z \in Z : r \leq d(z, V) \leq R \}.
	\]
\end{definition}

If an arc passes through (or close to) a parabolic point in the boundary, we want to reroute it
around that point.  The following definition will be used frequently in the following two sections.
\begin{definition}[\cite{Mac-08-quasi-arc}]\label{def-iota-follows}
	For any $x$ and $y$ in an embedded arc $A$, let $A[x,y]$ be the closed, possibly trivial,
	subarc of $A$ that lies between them.
	
	An arc $B$ {\em $\iota$-follows} an arc $A$, for some $\iota \geq 0$, if there
	exists a (not necessarily continuous) map $p:B \rightarrow A$, sending endpoints to endpoints,
	such that for all $x,\,y \in B$, $B[x,y]$ is in the $\iota$-neighbourhood of
	$A[p(x),p(y)]$; in particular, $p$ displaces points at most $\iota$.
\end{definition}

We now define our notion of avoidable set, which is a quantitatively controlled version
of the no local cut point and not locally disconnecting conditions.
\begin{definition}\label{def-avoidable-set}
	Suppose $(X,d)$ is a complete, connected metric space.
	A set $V \subset X$ is \emph{$L$-avoidable on scales below $\delta$} for $L \geq 1$, $\delta \in (0, \infty]$
	if for any $r \in (0, \delta/2L)$,
	whenever there is an arc $I \subset X$ and points $x, y \in I \cap A(V,r,2r)$ so that
	$I[x,y] \subset N(V, 2r)$,
	there exists an arc $J \subset A(V,r/L,2rL)$ with endpoints $x, y$ so that
	$J$ $(4rL)$-follows $I[x,y]$.
\end{definition}

The goal of this section is the following proposition.
\begin{proposition}\label{prop-avoidable-bdry}
	Let $(G, \cP_1)$ and $(G, \cP_1 \cup \cP_2)$ be relatively hyperbolic groups, 
	where all groups in $\cP_2$ are proper infinite hyperbolic subgroups of $G$ ($\cP_2$ may be empty), 
	and all groups in $\cP_1$ are proper, finitely presented and one-ended.  
	Let $X = X(G, \cP_1)$, and let $\cH$ be the collection of
	all horoballs of $X$ and left cosets of the subgroups of $\cP_2$. (As usual we regard $G$ as a subspace of $X$.)
	
	Suppose that $\bdry X$ is connected and locally connected, with no global cut points.
	Suppose that $\bdry P$ does not locally disconnect $\bdry X$ for each $P\in\cP_2$.	
	Then there exists $L \geq 1$ so that for every $H \in \cH$, $\bdry H \subset \bdry X$ is
	$L$-avoidable on scales below $e^{-\eps d(w,H)}$.
\end{proposition}
This proposition is proved in the following two subsections.

\subsection{Avoiding parabolic points}
We prove Proposition~\ref{prop-avoidable-bdry} in the case $H$ is a horoball.
This is the content of the following proposition.
\begin{proposition}\label{prop-thick-parabolics}
	Suppose $(G, \cP)$ is relatively hyperbolic, $\bdry (G,\cP)$ is connected and locally connected with no global cut points,
	and all peripheral subgroups are one-ended and finitely presented.
	Then there exists $L \geq 1$ so that for any horoball $O \in \cO$,
	$a_O = \bdry O \in \bdry (G, \cP)$ is $L$-avoidable on scales below $e^{-\eps d(w, O)}$.
\end{proposition}
The reason for restricting to this scale is that this is where the geometry of the boundary is determined
by the geometry of the peripheral subgroup.
Recall from Proposition~\ref{prop-bdry-nice} that such parabolic points are not local cut points.

The first step is the following simple lemma about finitely presented, one-ended groups.
It essentially states that we can join two large elements of such a group without going too close or too far from the identity.
Near a parabolic point, this allows us to prove Proposition~\ref{prop-thick-parabolics} by joining two suitable points
without going to far from or close to the parabolic point.

\begin{lemma}\label{lem-fin-pres-annuli}
	Suppose $P$ is a finitely generated, one-ended group, given
	by a (finite) presentation where all relators have
	length at most $M$, and let $\Gamma(P)$ be its Cayley graph.
	Then any two points $x, y \in \Gamma(P)$ such that $2M\leq r_x\leq r_y$, where $r_x=d(e,x)$ and $r_y=d(e,y)$,
	can be connected by an arc in $A(e, r_x/3, 2r_y) \subset \Gamma(P)$.
\end{lemma}

\begin{proof}
	By Lemma~\ref{ray}, we can find an infinite geodesic ray in $\Gamma(P)$ from $x$ which does not pass through
	$B(e, r_x/3)$.  Let $x'$ be the last point on this ray satisfying $d(e, x') = r_y$.
	Do the same for $y$, and let $y'$ denote the corresponding point.
	Note that $x'$ and $y'$ lie on the boundary of the unique unbounded component of
	$\{ z \in \Gamma(P) : d(e,z) \geq r_y \}$, which we denote by $Z$.
	We prove the lemma by finding a path from $x'$ to $y'$ contained in $A(e,r_y-M,r_y+M)$.
	
	Let $\beta_1$ be an arc joining $x'$ and $y'$ in $Z$. 
	It suffices to consider the case when $\beta_1 \cap \overline{B}(e, r_y) = \{x', y'\}$.
	Let $p$ be the first point of $[y', e]$ that meets $[e, x']$ in $\Gamma(P)$.
	Then the concatenation of $\beta_1, [y', p]$ and $[p, x']$ forms a simple, closed loop $\beta_2$ 
	in $\Gamma(P)$.
	
	As $\beta_2$ represents the identity in $P$, there exists a diagram $\cD$ for $\beta_2$: a connected, simply connected,
	planar $2$-complex $\cD$ together with a map of $\cD$ into the Cayley complex $\Gamma^2(P)$
	sending cells to cells and $\partial \cD$ to $\beta_2$.
	
	Let $\cD' \subset \cD$ be the union of closed faces $B \subset \cD$ which have a point $u \in \partial B$ with
	$d(u, e) = r_y$.  Let $\cD''$ be the connected component of $x'$ in $\cD'$.  
	Let $\gamma : \mathbb{S}^1 \ra \partial \cD''$ be the outer boundary path of $\cD'' \subset \R^2$.
	
	If either $\beta_1$ or $[y',p]\cup[p,x']$ live in $A(p, r_y-M, r_y+M)$, we are done.
	Otherwise, as we travel around $\gamma$ from $x'$, in one direction we must take a value $>r_y$, and in the other a value $< r_y$,
	thus there is a point $v \in \gamma \setminus \{x'\}$ with $d(e, v) = r_y$.
	If $v$ is in the interior of $\cD$, the adjacent faces are in $\cD'$, giving a contradiction.  So $v \in \beta_2$, and
	$v$ must be $y'$.  Thus there is a path from $x'$ to $y'$ in $\cD' \subset A(e, r_y-M, r_y+M)$.
\end{proof}

We can now prove the proposition.  The idea is similar to Lemma~\ref{lem-bdry-chains}, Case~1:
we find a suitable path using Lemma~\ref{lem-fin-pres-annuli} and push it out to $\bdry X$.
\begin{proof}[Proof of Proposition~\ref{prop-thick-parabolics}]
	By Proposition~\ref{prop-bdry-nice}, parabolic points in the boundary are not local cut points.
	
	We claim that there exists an $L \geq 1$ so that for any parabolic point $a_O$,
	any $r \leq e^{-\eps d(w, O)}/2L$, and any $a, b \in A(a_O, r, 2r)$, there exists
	an arc $J \subset A(a_O, r/L, 2rL)$ joining $a$ to $b$.  
	
	This claim suffices to prove the proposition, because 
	the $4rL$-following property automatically follows from $\diam(B(a_O, 2rL)) \leq 4rL$.
	We now proceed to prove the claim.
	
	Using the notation of Lemma~\ref{prodhorrel}, let $gP$ be the left coset of a peripheral group that corresponds to $O$, 
	and let $q \in gP$ be the first point of $[w,a_O)$ in $\overline{O}$.
	Recall that $d_O = d(w, O) \approx d(w, q)$.
	Let $q_a, q_b$ be the last points of $(a_O, a), (a_O, b)$ contained in $\overline{O}$.
		
	We begin by describing the positions of $q$, $q_a$ and $q_b$ in the path metric $d_P$ on $gP \subset X$.
	We write $x\asymp_C y$ if the quantities $x,y$ satisfy $x/C\leq y\leq Cx$. 
	
	Since
	\[
		e^{-\eps (a|a_O)} \asymp_{C_0} \rho(a, a_O) \asymp_2 r \leq e^{-\eps d_O} /2L,
	\]
	we have, for some $C_1 = C_1(C_0, \eps)$,
	\[
		(a | a_O) \approx_{C_1} \log(r^{-1/\eps}) \geq d_O + \log(2L)/\eps,
	\]
	so for $L \geq L(C_1,\eps)$, we have $(a | a_O) \geq d_O$, and likewise $(b | a_O) \geq d_O$.
	
	Lemmas \ref{lem-horoball-dist2} and \ref{prodhorrel}(1), with $A$ and $E$ as before, give
	\begin{multline}\label{eq-avdpar1}
		2 \log (d_P(q_a, q)) \approx_A d(q_a, q) \approx_{2E} 2((a | a_O) - d_O) \\
			\approx_{2C_1} 2 \log(r^{-1/\eps} e^{-d_O}) \geq 2 \log(2L)/\eps,
	\end{multline}
	so 
	$d_P(q_a, q) \asymp_{C_2} r^{-1/\eps} e^{-d_O}$, for $C_2 = C_2(A, E, C_1)$.
	Let $r_P$ be the smaller of $d_{P}(q_a, q), d_{P}(q_b, q)$,
	and notice that the larger value is at most $C_2^2 r_P$.
	
	We now use Lemma~\ref{lem-fin-pres-annuli} to find a chain of points
	$q_a=q_0, q_1, \ldots, q_n=q_b$ in $gP$ joining $q_a$ to $q_b$ in $gP$,
	so that $q_i \in A(q, r_P/3, 2r_P C_2^2 )$ in the metric $d_P$.
	Let $c_i = q_i q_0^{-1} a$, for $i=0, \ldots, n-1$, and set $c_n = b$.
	This gives a chain of points $\{c_i\}$ in $\bdry X$ joining $a$ to $b$.
	(Observe that we can take $(a_O, c_i) = q_iq_0^{-1}(a_O, a)$.)
	
	Lemma~\ref{lem-horoball-dist2} and \eqref{eq-avdpar1} imply that
	\[
		d(q_i,q) \approx 2 \log (d_P(q_i, q)) \approx 2 \log (d_P(q_a, q)) \approx d(q_a,q) \gtrsim 2\log(2L)/\eps,
	\]
	with total error $C_3 = C_3(A, E, C_1, C_2)$.
	So if $L \geq L(C_3, \eps)$, we have $d(q_i, q) > E$, and therefore 
	$(c_i | a_O) \geq d_O$ by Lemma~\ref{prodhorrel}(3).
	
	Now Lemma~\ref{prodhorrel}(1) shows that 
	\[
		(c_i | a_O) \approx_E d(q_i, q)/2 + d_O \approx_{C_3} d(q_a, q)/2 + d_O \approx_E (a | a_O),
	\]
	so $\rho(c_i, a_O) \asymp \rho(a, a_O) \asymp r$, with total error $C_4 = C_4(C_3, E)$,
	that is $c_i \in A(a_O, r/C_4, C_4 r)$.
	
	We now wish to join each $c_i$ and $c_{i+1}$ in a suitable annulus around $a_O$.
	Consider the geodesics between $w$, $c_i$ and $c_{i+1}$, and observe that
	$d(q_i, q) > E > \del_X$, and $d_P(q_i, q_{i+1}) = 1$.  From this, and \eqref{eq-avdpar1}, we see that
	\[  
		(c_{i} | c_{i+1}) \gtrsim d(w, q_i) \approx d_O + d(q, q_a) \approx 
			2 \log(r^{-1/\eps}) -d_O,
	\]
	and so, for suitable $C_5$,
	\[
		\rho(c_i, c_{i+1}) \leq C_0 e^{-\eps (c_i | c_{i+1})}
			\leq C_{5} r^2 e^{\eps d_O} \leq C_{5} r/2L.
	\]
	
	By Proposition~\ref{prop-bdry-lin-conn}, $\bdry X$ is $L'$-linearly connected,
	so if $L \geq C_4 C_5 L'$ 
	we can join $c_i$ to $c_{i+1}$ in $B(c_i, L' C_5 r/2L) \subset B(c_i, r/2C_{4})$.
	Since $c_i \in A(a_O, r/C_4, C_4r)$ 
	we have joined $c_i$ to $c_{i+1}$ in
	$A(a_O, r/2C_{4}, 2C_4 r)$, and the claim follows.
\end{proof}

\subsection{Avoiding hyperbolic subgroups}

In this section we complete the proof of Proposition~\ref{prop-avoidable-bdry} for $H = gP$, where
$P \in \cP_2$.  By assumption, $\bdry H \subset \bdry X$ does not locally disconnect $\bdry X$.

First, we show that boundaries of peripheral groups are porous.
\begin{definition}[{e.g.\ \cite[14.31]{Hei-01-lect-analysis}}]\label{def-porous}
	A set $V$ in a metric space $(Z,\rho)$ is \emph{$C$-porous on scales below $\delta$} if for any $z \in V$ and $0 < r < \delta$,
	there exists $z' \in B(z, r)$ so that $\rho(z', V) \geq r/C$.
\end{definition}

\begin{lemma}\label{lem-bdry-porous}
	Under the assumptions of Proposition~\ref{prop-avoidable-bdry},
	there exists $L_1$ so that for every $H \in \cH$, 
	$\bdry H \subset \bdry X$ is $L_1$-porous on scales below $e^{-\eps d(w, H)}$.
\end{lemma}
The proof follows from the partial self-similarity of Corollary~\ref{cor-self-sim2} and the fact
that for any $H \in \cH$, $\bdry H$ has empty interior in $\bdry X$.
\begin{proof}
	Observe that if $H$ is a horoball, then $\bdry H$ is a point in a connected space, and so is automatically porous.
	
	If the conclusion is false, we can find a sequence of cosets $H_n = g_n P_n$, for $P_n \in \cP_2$,
	points $a_n \in \bdry H_n$ and values 
	$r_n \leq e^{-\eps d(w, H_n)}$ so that $N(\bdry H_n, r_n/n) \supset B(a_n, r_n)$.
	
	Let $d_{r_n} \approx \log(r_n^{-1/\eps})$ be given as in Corollary~\ref{cor-self-sim2} with $z=a_n$, $r=r_n$.
	Assume that we can take a subsequence and reindex so that $d_{r_n} \geq 0$ for all $n$.  Let $x_n \in [w, a_n)$ be the point satisfying $d(w, x_n) = d_{r_n}$.
	Every $H \in \cH$ is uniformly quasi-convex, see Lemma~\ref{lem-x-asymp-tree-graded}(2),
	and $r_n \leq e^{-\eps d(w, H_n)}$, so $d(x_n, H_n)$ is uniformly bounded for any such $x_n$.
	Therefore there exists $h_n \in G$ so that $d(h_n^{-1} w, x_n) \leq D$, for some uniform constant $D$.
	
	Thus Corollary~\ref{cor-self-sim2} implies that there exists $L_0 = L_0(D)$ and $L_0$-bi-Lipschitz
	maps $f_n: (B(a_n, r_n), \frac{1}{r_n}\rho) \ra \bdry X$ 
	induced by the action of $h_n$, so that $B(f_n(a_n), 1/L_0) \subset f_n(B(a_n, r_n))$.
	
	As $h_n H_n = H_n'$ for some $H_n' \in \cH$, and $d(H_n', w)$ is uniformly bounded, we may take a subsequence so that
	$H_n' = H' \in \cH$ for all $n$, and moreover that $f_n(a_n) \in \bdry H'$ converges to $a \in \bdry H'$.
	Therefore, for all sufficiently large $n$,
	\begin{align*}
		B(a, 1/2L_0) & \subset B(f_n(a_n), 1/L_0)
			\subset f_n(B(a_n, r_n)) \\
			& = f_n(B(a_n, r_n) \cap N(\bdry H_n, r_n/n))
			\subset N(\bdry H', L_0/n),
	\end{align*}
	so $a$ is in the interior of $\bdry H' \subset \bdry X$, since $\bdry H'$ is closed 
	in $\bdry X$.  
	This is a contradiction because $\bdry H'$ is not all of $\bdry X$ 
	(proper peripheral subgroups of a relatively hyperbolic group are of infinite index),
so if $a$ is a point of $\bdry H'$, one can use the action of $H'$ to
find points in $\bdry X \setminus \bdry H'$ that are arbitrarily close to
$a$.

	There remains the case where infinitely many $d_{r_n} < 0$.
	But then for such a subsequence we have all $r_n > C > 0$, and $d(e, H_n)$ is uniformly bounded.  Therefore we can proceed as above to take a subsequence so that $H_n=H'$ for all $n$ and $B(f_n(a_n), C) \subset N(\bdry H_n, r_n/n) \subset N(\bdry H', 1/n)$.  The rest of the argument is the same.
\end{proof}

We continue with the proof of Proposition~\ref{prop-avoidable-bdry}.
The basic idea is to use partial self-similarity and a compactness argument
to upgrade the topological condition of not locally disconnecting
to the quantitative $L$-avoidable condition.

We begin with the following lemma.
\begin{lemma}\label{lem-avoid-hyp-1}
	Given $L_1\geq 1$, there exists $L_2=L_2(X, L_1)$ independent of $H = gP$, $P \in \cP_2$,
	so that for any $r \leq e^{-\eps d(w,H)}/L_2$, 
	and any two points $u, v \in A(\bdry H, r/L_1, 2r)$ so that $\rho(u, v) \leq 4r$, 
	there exists an arc 
\[ K \subset A(\bdry H, r/L_2, 2L_2r) \] joining $u$ to $v$ with $\diam(K) \leq 2L_2 r$.
\end{lemma}
\begin{proof}
	As in Lemma~\ref{lem-bdry-porous}, we assume the conclusion is false, and will use self-similarity 
	to derive a contradiction.
	
	If the conclusion is false, there is a sequence of $H_n = g_n P_n$ with $P_n \in \cP_2$,
	$r_n \leq e^{-\eps d(w, H_n)}/n$, and points $a_n \in \bdry H_n$ and 
	$u_n, v_n \in B(a_n, 6r_n) \cap A(\bdry H_n, r_n/L_1, 2r_n)$
	so that there is no arc of diameter at most $2nr_n$ joining $u_n$ to $v_n$ in $A(\bdry H_n, r_n/n, 2nr_n)$.
	
	As before, the geodesic $[w, a_n)$ essentially travels from $w$ straight to $H_n$ then along $H_n$ to $a_n \in \bdry H_n$.
	More precisely, there are constants $C_1$ and $D$ depending on the uniform quasi-convexity constant of $H_n$ and $\del_X$ so that
	for any $r \leq e^{-\eps d(w, H_n)}/C_1$, the point $x$, defined by 
	Corollary~\ref{cor-self-sim2}(1) applied to $z = a_n$ and $r$, lies within distance $D$ of $Gw$.
	Let $L_0=L_0(D)$ be the corresponding constant from Corollary~\ref{cor-self-sim2}.
	
	Let $L'$ be the linear connectivity constant of $\bdry X$, and set $r_n' = 10L_0^2L' r_n$.
	For $n$ large enough, $r_n' \leq e^{-\eps d(w, H_n)}/C_1$, and so we find $h_n \in G$ that induces
	a $L_0$-bi-Lipschitz map $f_n : (B(a_n, r_n'), \frac{1}{r_n'}\rho) \ra \bdry X$, with
	$B(f_n(a_n), 1/L_0) \subset f_n(B(a_n, r_n'))$.  Note that $h_n$ maps $H_n$ to some $H_n' \in \cH$, with $f_n(a_n) \in \bdry H_n'$.
	As $d(w, H_n')$ is uniformly bounded, we can take a subsequence so that $H_n' = H' \in \cH$.
	
	The images $f_n(u_n), f_n(v_n)$ lie in $B(f_n(a_n), T ) \setminus N(\bdry H', t)$, where
	 $T = 6r_n L_0/r_n' < 1/L_0L'$ and $t = r_n/r_n'L_1L_0$ are independent of $n$.
	Let
	\[
		W = \{ (u, v, a) : a \in \bdry H', \{u, v\} \subset \overline{B}(a, T) \setminus N(\bdry H', t) \} \subset (\bdry X)^3,
	\]
	and define $f: W \ra (0, T]$ to be supremal so that for $(u,v,a) \in W$ there exists an arc
	joining $u$ to $v$ in $B(a, 1/L_0) \setminus N(\bdry H', f(u,v,a))$.
	We know that $f$ is positive because for any $(u,v,a) \in W$, $\overline{B}(a, T)$ lies in a connected open set $U \subset B(a, 1/L_0)$,
	and as $\bdry H'$ does not locally disconnect, $U \setminus \bdry H'$ is connected and we can join $u$ to $v$ in this set.
	
	Observe that by local connectivity $f$ is continuous, and $W$ is compact, so $f(u,v,a) \geq 2C_2 > 0$ for some $C_2$ and all $(u,v,a) \in W$.
	
	Now $(f_n(u_n), f_n(v_n), f_n(a_n)) \in W$, so there exists an arc $K$ joining $f_n(u_n)$ to $f_n(v_n)$ with
	$K' \subset B(f_n(a_n), 1/L_0) \setminus N(\bdry H', C_2)$.
	The preimage $K = f_n^{-1}(K')$ joins $u_n$ to $v_n$ so that 
	\[ K \subset B(a_n, r_n') \cap A(\bdry H, r_n'C_2/L_0, r_n'), \]
	which is a contradiction for large enough $n$.
\end{proof}

We now finish the proof of Proposition~\ref{prop-avoidable-bdry},
fixing constants $L_1\geq 2$ from Lemma~\ref{lem-bdry-porous} and $L_2 = L_2(X, L_1)$ from Lemma~\ref{lem-avoid-hyp-1}.

Let $H = gP$, $P \in \cP_2$ be fixed.
Suppose we are given $r \leq e^{-\eps d(w, H)}/L_2$, points $x, y \in A(\bdry H, r, 2r)$,
and an arc $I \subset N(H, 2r)$ with endpoints $x$ and $y$.

We build our desired arc $J$ from $x$ to $y$ in stages.
First, let $x=z_0', z_1', \ldots, z_m' = y$ be a (finite) chain of points that $0$-follows $I[x,y]$, so that
$\rho(z_i, z_{i+1}) \leq r$.
(The definition of $\iota$-follows is extended from arcs to chains in the obvious way.)
For each $i$, if $\rho(z_i', \bdry H) \leq r/L_1$, use
Lemma~\ref{lem-bdry-porous} to find a point $z_i$ at most $r/L_1+r$ away from $z'_i$, and outside
$N(\bdry H, r/L_1$).  Otherwise let $z_i = z_i'$.  

This new chain satisfies $\{z_i\} \subset A(\bdry H, r/L_1, 2r)$, and
for every $i$, $\rho(z_i, z_{i+1}) \leq r+2(r/L_1+r) \leq 4 r$.
It $2r$-follows $\{z_i'\}$, and thus $\{z_i\}$ also $2r$-follows $I$.

By Lemma~\ref{lem-avoid-hyp-1} for each $i$ there exist an arc $J_i$ joining $z_i$ and $z_{i+1}$ 
which lies in $A(\bdry H, r/L_2, 2L_2 r)$ and has $\diam(J_i) \leq 2L_2 r$.
From this, we extract an arc $J$ by cutting out loops: travel along $J_0$ until you meet $J_j$
for some $j \geq 1$, and at that point cut out the rest of $J_0$ and all $J_k$ for $1 \leq k < j$.
Concatenate the remainders of $J_0$ and $J_j$ together, and continue along $J_j$.

The resulting arc $J$ will $2L_2 r$-follow the chain $\{z_i \}$, and so it will $4L_2 r$-follow
$I$ as desired.\qed

\section{Quasi-arcs that avoid obstacles}\label{sec-qarcs-that-avoid}

A quasi-arc is a metric space which is quasisymmetrically homeomorphic
to $[0,1]$ with its usual metric.
Tukia and V\"ais\"al\"a showed that one can equivalently define a quasi-arc 
as a metric space which is a topological arc, and which is doubling and
linearly connected~\cite[Theorem 4.9]{TV-80-qs}.  (If this arc is
$\lambda$-linearly connected, we call the arc a \emph{$\lambda$-quasi-arc}.)

As discussed in the introduction, Tukia showed that doubling and linearly
connected metric spaces contain quasi-arcs joining any two 
points~\cite[Theorem 1A]{Tuk-96-qarc}.

In this section 
we build quasi-arcs in a metric space that avoid specified obstacles.
This result can be viewed from the perspective of Diophantine approximation
for finite volume hyperbolic manifolds; see Example~\ref{ex-avoiding-dioph-approx}.
The methods we use build on the alternative proof of Tukia's theorem
found in \cite{Mac-08-quasi-arc}.

\subsection{Collections of obstacles}
The next definition gives us control on a collection of obstacles.
\begin{definition}
	Let $(Z, \rho)$ be a compact metric space.
	Let $\cV$ be a collection of compact subsets of $Z$ provided with some map 
	$D: \cV \ra (0, \infty)$, which we call a \emph{scale function}.
	
	The \emph{(modified) relative distance function} $\Delta: \cV \times \cV \ra [0, \infty)$ is
	defined	for $V_1, V_2 \in \cV$ as
	\[
		\Delta(V_1, V_2) = \frac{\rho(V_1,V_2)}{\min\{D(V_1),D(V_2)\}}.
	\]
	We say $\cV$ is \emph{$L$-separated} if for all $V_1, V_2 \in \cV$, if $V_1 \neq V_2$ then
	$\Delta(V_1, V_2) \geq \frac1L$.
\end{definition}

	As we saw in Section~\ref{sec-avoiding}, we often only have control on topology on a 
	sufficiently small scale.
	The purpose of the scale function is to determine the size of the neighbourhood of each 
	$V \in \cV$ on which we have this control.
	An example of a scale function is $D(V) = \diam(V)$, if every $V \in \cV$ has $|V| > 1$.
	In this case, $\Delta$ is precisely the usual relative distance function, 
	e.g.\ \cite[page 59]{Hei-01-lect-analysis}.
	
	The goal of this section is the following result.

\begin{theorem}\label{thm-modified-qarc}
	Let $(Z, \rho)$ be an $N$-doubling, $L$-linearly connected, compact metric space, and $L \geq 10, N \geq 1$ constants.
	Suppose $\cV$ is an $L$-separated collection of compact subsets of $Z$ 
	with scale function $D:\cV \ra (0, \infty)$,
	so that $V \in \cV$ is $L$-porous and $L$-avoidable on scales below $D(V)$ 
	(see Definitions~\ref{def-porous} and \ref{def-avoidable-set}).
	
	For any $\nu \geq 1$ there exists a constant $\lambda = \lambda(N, L, \nu)$
	so that given any two points $x, y \in Z$,
	if for all $V \in \cV$ we have $\rho(\{x,y\},V) \geq D(V)/\nu$,
	then $x$ and $y$ can be joined by a $\lambda$-quasi-arc 
	which satisfies $\rho(\gamma, V) \geq \frac1\lambda D(V)$ for each $V \in \cV$.
\end{theorem}
The following result shows that such endpoints exist.
\begin{proposition}\label{prop-porous-points}
	Let $(Z, \rho)$ be a compact metric space, and $L \geq 10$, $N \geq 1$ constants.
	Suppose $\cV$ is an $L$-separated collection of compact subsets of $Z$ 
	with scale function $D:\cV \ra (0, \infty)$,
	and suppose that each $V \in \cV$ is $L$-porous on scales below $D(V)$.
	
	For any $r \leq 1/5L$, given $p \in Z$ there exists $q \in Z$ so that
	$\rho(p,q) \leq \diam(Z) r$, and that 
	for all $V \in \cV$ we have $\rho(q, V) \geq D(V) r / 8L^2$.
\end{proposition}

For the remainder of this paper we will use the following 
corollary to Theorem~\ref{thm-modified-qarc}.
\begin{corollary}\label{cor-avoiding-qarc}
	Let $(Z, \rho)$ be a compact, $N$-doubling and $L$-linearly connected metric space.
	Suppose $\cV$ is an $L$-separated collection of compact subsets of $Z$ with scale 
	function $D : \cV \ra (0, \infty)$, and that each $V \in \cV$ is both $L$-porous and 
	$L$-avoidable on scales below $D(V)$.
	
	Then for a constant $\lam = \lam(N,L)$ there exists a $\lam$-quasi-arc $\gam$ in $Z$
	which satisfies $\diam(\gam) \geq \frac12 \diam(Z)$, and 
	$\rho(\gam,V) \geq \frac{1}{\lam} D(V)$ for each $V \in \cV$.
\end{corollary}
\begin{proof}[Proof of Corollary~\ref{cor-avoiding-qarc}]
	Let $x'$ and $y'$ be two points at maximum distance in $Z$.
	Apply Proposition~\ref{prop-porous-points} with $r=1/5L$ to $x'$ and $y'$ 
	to find points $x$ and $y$ which are at least $\diam(Z)/2$ apart, 
	and have $\rho(\{x,y\},V) \geq D(V)/40L^3$ for any $V \in \cV$.
	The corollary then follows from Theorem~\ref{thm-modified-qarc}.
\end{proof}

Two simple applications of Corollary~\ref{cor-avoiding-qarc} are the following.
\begin{example}
	Let $Z$ be the usual square Sierpi\'nski carpet in the plane, with Euclidean metric $d_{Euc}$,
	and let $\cV$ be the set of
	peripheral squares, i.e., boundaries of $[0,1]^2$, $[1/3,2/3]^2$, and so on.
	Define $D(V) = \diam(V)$ for each $V \in \cV$.
	
	The assumptions of Corollary~\ref{cor-avoiding-qarc} are satisfied for suitable $N$ and $L$,
	so there exists some $\lambda$ and a $\lambda$-quasi-arc $\gamma$ in $Z$ which satisfies
	$d_{Euc}(\gamma, V) \geq \frac{1}{\lambda} \diam(V)$ for each $V \in \cV$.
\end{example}
It is not immediately obvious that there exists a point satisfying this last separation 
condition, let alone a quasi-arc,
although in the carpet it is possible to build such an arc by hand.
\begin{example}\label{ex-avoiding-dioph-approx}
	Let $M$ be a finite volume hyperbolic $n$-manifold, with $n \geq 3$,
	and a choice of base point $p \in M$.
	The universal cover of $M$ is $\tilde{M} = \HH^n$, and fix a
	lift $\tilde{p}$ of $p$.  Let $\cH$ be a collection of horoballs for
	the action $\pi_1(M) \curvearrowright \tilde{M}=\HH^n$.
	
	Let $Z = \bdry \HH^n = \Sph^{n-1}$, with $\pi_1(M)$ acting on $Z$.
	Let $\cV$ be the collection of parabolic points $\bdry H \in Z$, for $H \in \cH$,
	with scale function $D(H) = e^{-d(\tilde{p},H)}$.
	
	Points in $\Sph^{n-1}$ are avoidable and porous, and $\Sph^{n-1}$ is
	both doubling and linearly connected.
	The linear separation of $\cV$ follows from Lemma~\ref{lem-parabolic-points}.
	Theorem~\ref{thm-modified-qarc} applies to find many quasi-arcs in $\Sph^{n-1}$,
	which do not go too close to parabolic points.
	Moreover, geodesic rays from $\tilde{p}$ to these quasi-arcs do not go far into
	horoballs by Lemma~\ref{farfrombpoints}.
	
	Identifying $\Sph^{n-1}$ with the tangent sphere $T_p M$, this means at
	any point $p \in M$ we can find a compact subset $K \subset M$ so that
	there are lots of (quasi-arc) paths of directions in $T_p M$ with the geodesic
	rays in these directions living in $K$.
\end{example}

\subsection{Building the quasi-arc}

The way that Theorem~\ref{thm-modified-qarc} builds a quasi-arc is by an inductive process: 
starting with any arc in $Z$, push the arc away from the largest obstacles in $\cV$, 
then push it away from the next largest, and so on.  
While this is going on, one also ``cuts out loops'' in order to ensure the limit arc is a quasi-arc.
There is some delicacy involved in making sure the constants work out correctly.

As a warm-up, we show how to find points far from obstacles.
\begin{proof}[Proof of Proposition~\ref{prop-porous-points}]
	If $\cV = \emptyset$, the result is trivial.
	Otherwise, let $D_0 = \sup\{D(V):V \in \cV\}$.
	Observe that as every $V \in \cV$ is $L$-porous, we have $D_0 \leq L \diam(Z)$. 
	
	We filter $\cV$ according to size.
	For $n \in \N$, let 
	$\cV_n = \{ V \in \cV : r^n < D(V)/D_0 \leq r^{n-1} \}$, 
	where $r \leq 1/5L \leq 1/50$ is given.
	(Recall that $L \geq 10$.)
	Note that $N(V, D_0r^n/2L) \cap N(V', D_0r^n/2L) = \emptyset$ if
	$V$ and $V'$ are distinct elements of $V_n$, because
	$\cV$ is $L$-separated.
	
	Let $x_0 = p$, and proceed by induction on $n \in \N$.
	Suppose $\rho(x_{n-1},V) \leq D_0r^n/4L$ for some (unique) $V \in \cV_n$.
	Then as $V$ is $L$-porous on scales below $D(V) > D_0r^n/4L$,
	we can find $x_n \in Z$ so that 
	$\rho(x_{n-1}, x_n) \leq D_0r^n/4L+D_0r^n/4L=D_0r^n/2L$
	and $\rho(x_n,V) \in [ D_0r^n/4L^2, D_0r^n/4L]$.
	For any other $V' \in \cV_n$, we have
	\begin{align*}
		\rho(x_n, V') \geq \rho(V,V') - \rho(V, x_n) 
			\geq \frac{D_0r^n}{L}-\frac{D_0r^n}{4L} > \frac{D_0 r^n}{4L^2}.
	\end{align*}
	If no such $V$ exists, set $x_n = x_{n-1}$.  In either case, for all
	$V' \in \cV_n$, we have $\rho(x_n, V') \geq D_0r^n/4L^2.$
	
	The sequence $\{x_n\}$ converges to a limit $q$.
	Observe that for any $n \geq 0$,
	\begin{align*}
		\rho(x_n, q) & \leq \rho(x_n, x_{n+1}) + \rho(x_{n+1},x_{n+2}) + \cdots \\ 
			& \leq \frac{D_0r^{n+1}}{2L}+\frac{D_0r^{n+2}}{2L}+ \cdots
			= \frac{D_0 r^{n+1} }{2L(1-r)}.
	\end{align*}
	In particular, $\rho(p,q) = \rho(x_0,q) \leq D_0 r/2L(1-r) \leq \diam(Z)r$.
	
	For any $V \in \cV$, there exists $n$ so that $V \in \cV_n$, and we have
	\begin{align*}
		\rho(q,V) & \geq \rho(x_n, V) - \rho(x_n, q) 
			 \geq \frac{D_0 r^n}{4L^2} - \frac{D_0 r^{n+1}}{2L(1-r)} \\
			& = \frac{D_0 r^{n-1} r}{4L^2} \left(1-\frac{2rL}{1-r} \right) 
			 \geq \frac{D(V) r}{8L^2},
	\end{align*}
	where we used that $2rL/(1-r) \leq 1/2$.
\end{proof}

To find quasi-arcs, we need more machinery.
We now recall some terminology and results from \cite{Mac-08-quasi-arc}.
An arc $A$ in a doubling and complete metric space is an
{\em $\iota$-local $\lambda$-quasi-arc} if
$\diam(A[x,y]) \leq \lambda d(x,y)$ for all $x,\, y \in A$ such that
$d(x,y) \leq \iota$.  (See Definition~\ref{def-iota-follows} for the notion of
$\iota$-following.)
\begin{remark}\label{rmk-local-global-quasiarc}
	Any $\iota$-local $\lambda$-quasi-arc $\gamma$ is a $\lambda'$-quasi-arc with
	$\lambda' = \max\{ \lambda, \diam(\gamma)/\iota\}$.
\end{remark}

\begin{proposition}[{\cite[Proposition 2.1]{Mac-08-quasi-arc}}]\label{prop-coarsestr}
 Given a complete metric space $(Z, \rho)$ that is
 $L$-linearly connected and $N$-doubling, there exist constants
 $s=s(L,N)>0$ and $S=S(L,N)>0$ with the following property:
 for each $\iota > 0$ and each arc $A \subset X$, there exists an arc
 $J$ that $\iota$-follows $A$, has the same endpoints as $A$,
 and satisfies
 \begin{equation} \label{eq-cqa}
  \forall u,v \in J,\  \rho(u,v) < s\iota \implies
   \diam(J[u,v]) < S\iota.
 \end{equation}
\end{proposition}

\begin{lemma}[{\cite[Lemma 2.2]{Mac-08-quasi-arc}}]\label{lem-approx}
 Suppose $(Z,\rho)$ is an $L$-linearly connected, $N$-doubling, complete
 metric space, and let  $s,\, S,\, \eps$ and $\del$
 be fixed positive constants satisfying
 $\del \leq \min\{\frac{s}{4+2S},\frac{1}{10}\}$.
 Now, if we have a sequence of arcs $J_0, J_1, \ldots, J_n, \ldots$ in $Z$,
  such that for every $n \geq 1$
 \begin{itemize}
  \item $J_{n}$ $\eps \del^n$-follows $J_{n-1}$, and
  \item $J_{n}$ satisfies \eqref{eq-cqa} with
   $\iota = \eps \del^n$ and $s,\,S$ as fixed above,
 \end{itemize}
 then the Hausdorff limit $J = \lim_\mathcal{H} J_n$
 exists, and is an $\eps \del^2$-local
 $\frac{4S+3\del}{\del^2}$-quasi-arc.
 Moreover, the endpoints of $J_n$ converge to the endpoints of $J$,
 and $J$ $\eps$-follows $J_0$.
\end{lemma}

We now use these results to build our desired quasi-arc.

\begin{proof}[Proof of Theorem \ref{thm-modified-qarc}]
	As in the proof of Proposition~\ref{prop-porous-points},
	let $D_0 = \sup\{D(V):V \in \cV\}$; if $\cV = \emptyset$, set $D_0 = \diam(Z)$.
	Recall that $D_0 \leq L \diam(Z)$, and we assume that $L \geq 10$.
	
	Let $r=r(L, N, \nu) > 0$ be fixed sufficiently small as determined later in the proof.
	As before, define $\cV_n = \{ V \in \cV : r^n < D(V)/D_0 \leq r^{n-1} \}$, for $n \in \N$,
	and let $\cC_n = \{ N(V, D_0r^n/4L) : V \in \cV_n \}$.
	As $\cV$ is $L$-separated, each $\cC_n$ consists of disjoint neighbourhoods.
	(Note that two neighbourhoods from different $\cC_n$ may well intersect.)
	
	Suppose $x$ and $y$ are given with $\rho(x, V), \rho(y,V) \geq D(V)/\nu$
	for each $V \in \cV$.
	Without loss of generality, we assume that $\nu \geq 10L^2$.
	Let $M \in \Z$ be maximal so that $\rho(x,y) < D_0 r^M / 4L\nu$.
	We start with an arc $J_M = J_M[x,y]$ in $Z$ of diameter at most $L\rho(x,y)$, and
	build arcs $J_n$ in $Z$ by induction on $n > M$.
	
	\vspace{1mm}
	{\noindent\bf Inductive step: }	
	Assume we have been given an arc $J_{n-1} = J_{n-1}[x,y]$.
	
	First, assuming $n > 0$, 
	we modify $J_{n-1}$ independently inside the (disjoint) sets in $\cC_n$.
	Let $r'_n = D_0 r^n/2\nu$, and observe that for any $V \in \cV_n$,
	\[
		A(V, r'_n, 2r'_n) \subset A(V, r'_n/L, 5r'_nL) \subset N(V, D_0r^n/4L) \in \cC_n.
	\]
	
	Note that $x$ and $y$ lie outside $N(V,2r'_n)$, as 
	$ \rho(\{x,y\},V) \geq D(V)/\nu > D_0 r^n / \nu = 2r'_n$.
	
	Given $V \in \cV_n$, each time $J_{n-1}$ meets $N(V, r'_n/L)$, the arc $J_{n-1}$
	travels through
	$A=A(V, r'_n, 2r'_n)$ both before and after meeting $N(V, r'_n/L)$.
	For each such meeting, we use that $V$ is $L$-avoidable with ``$r$'' equal to $r'_n$
	to find a detour path in $A(V, r'_n/L, 2r'_nL)$ which $4r'_nL$-follows the previous path.
	After doing so, we concatenate the paths found into an arc $J_n'$, as at the end of the proof of Proposition~\ref{prop-avoidable-bdry}.
	This arc $J_n'$ will $4r'_nL$-follow $J_{n-1}$.
	
	If $n \leq 0$, set $J_n' = J_{n-1}$, which $0$-follows $J_{n-1}$.
	
	Second, apply Proposition~\ref{prop-coarsestr} to $J_n'$ with $\iota = r'_n/2L$.
	Call the resulting arc $J_n$: it $\iota$-follows $J_n'$, 
	so it $(D_0r^n/4L)$-follows $J_{n-1}$, as
	$\iota + 4r_n'L \leq 5r_n'L \leq D_0 r^n/4L$.
	Since $\rho(J'_n, V) \geq r'_n/L$ and $J_n$ $(r'_n/2L)$-follows $J'_n$,  
	we also have 
	\begin{equation}\label{eq-qarc1}
		\rho(J_n, V) \geq \frac{r'_n}{2L} = \frac{D_0r^n}{4L\nu}.
	\end{equation}
	
	\vspace{1mm}
	{\noindent\bf Limit arc: }	
	Consider the sequence of arcs $J_M, J_{M+1}, \ldots$.
	
	For every $i \in \N$, 
	$J_{M+i}$ $(D_0r^M/4L)r^i$-follows $J_{M+i-1}$.
	Let $s$ and $S$ be given by Proposition \ref{prop-coarsestr},
	then observe that $J_{M+i}$ satisfies $\forall u,v \in J,$
	\begin{multline*}
		\rho(u,v) < s\iota=\left(\frac{s}{\nu}\right)\left(\frac{D_0r^M}{4L}\right)r^i
		\\ \implies
      \diam(J[x,y]) < S\iota=\frac{S D_0 r^{M}}{4L\nu} \leq S \left(\frac{D_0r^M}{4L}\right)r^i.
	\end{multline*}
	In other words, $J_{M+i}$ satisfies \eqref{eq-cqa} with
	$s$ replaced by $s' = s/\nu$ and $\iota = (D_0r^M/4L)r^i$.
	
	We can assume that
	$r \leq \min\left\{\frac{s'}{4+2S},\frac{1}{10}\right\}$, 
	since $s'$ and $S$ depend only on $L$, $N$ and $\nu$.
	
	Now apply Lemma~\ref{lem-approx} to the arcs 
	$J_M, J_{M+1}, \ldots$ with $s'$ replacing $s$, 
	$\del = r$ and $\eps = D_0r^M/4L$,
	to find an arc $\gamma$, with endpoints $x$ and $y$.
	The arc $\gamma$ is a 
	$(D_0r^{M+2}/4L)$-local $\mu$-quasi-arc, where $\mu=\mu(L,N,\nu)$.
	
	For each $n \geq M$, $\gamma$ lies in a neighbourhood of $J_n$ of size at most
	\begin{equation}\label{eq-qarc3}
		\frac{D_0r^{n+1}}{4L} + \frac{D_0r^{n+2}}{4L} + \cdots = \frac{D_0r^{n+1}}{4L (1-r)}
			\leq \frac{D_0r^{n}}{8L\nu},
	\end{equation}
	where this last inequality holds for $r \leq 1/4\nu$.
	(We may now set $r = \min\{1/4\nu, s'/(4+2S)\}$.)

	In particular, $\gamma$ lies in a ball about $x$ of radius at most
	\begin{equation}\label{eq-qarc2}
		\diam(J_M) + \frac{D_0r^M}{8L\nu} 
			\leq L \rho(x,y)+\frac{D_0r^M}{8L\nu}
			\leq \frac{D_0r^M}{2\nu},
	\end{equation}
	so by Remark~\ref{rmk-local-global-quasiarc}, $\gamma$ is a $\lambda'$-quasi-arc
	with $\lambda'=\lambda'(L,N)$ the maximum of $\mu$ and
	$\diam(\gamma)(4L/D_0r^{M+2}) \leq 4L/\nu r^2$.

	\vspace{1mm}
	{\noindent\bf Avoiding obstacles: }
	For any $V \in \cV_n$ with $n > M$, \eqref{eq-qarc1} and \eqref{eq-qarc3} give
	\[
		\rho(\gamma, V) \geq \rho(J_n, V) - \frac{D_0r^n}{8L\nu} 
		\geq \frac{D_0r^n}{8L\nu} \geq \frac{r}{8L\nu} D(V).
	\]
	If $V \in \cV_n$ with $n \leq M$, then $\rho(x,V) \geq D(V)/\nu \geq D_0r^n/\nu$,
	while by \eqref{eq-qarc2} $\gamma$ lies in $B(x, D_0r^M/2\nu)$,
	so
	\[
		\rho(\gamma, V) \geq \frac{D_0r^n}{\nu} - \frac{D_0r^M}{2\nu} 
			\geq \frac{D_0r^n}{2\nu} \geq \frac{r}{2\nu}D(V).\qedhere
	\]
\end{proof}

\section{Building quasi-hyperbolic planes}\label{sec-build-planes}

We now have all we need to construct quasi-isometrically embedded hyperbolic planes.
\begin{proof}[Proof of Theorem \ref{thm-main-hyp-rel-hyp}]
Let $\cV = \{ \bdry H : H \in \cH \}$, where $\cH$ is the collection of
all horoballs of $X=X(G,\cP_1)$ and left cosets of the subgroups of $\cP_2$.
Define the scale function $D : \cV \ra (0, \infty)$ by $D(\bdry H) = e^{-\eps d(w, H)}$ for each $H \in \cH$.

The boundary $\bdry(G, \cP_1)$ is $N$-doubling, for some $N$, by Proposition \ref{prop-bdry-doubling}.

Theorem~\ref{thm-bowditch} implies that $\bdry (G,\cP_1)$ is connected and locally connected, with no global cut points.
By Proposition~\ref{prop-bdry-lin-conn} $\bdry (G,\cP_1)$ is $L_2$-linearly connected
for some $L_2 \geq 1$.  Proposition~\ref{prop-avoidable-bdry} implies that there exists $L_3 \geq 1$ so that for every $H \in \cH$, $\bdry H$
is $L_3$-avoidable on scales below $e^{-\eps d(w,H)}$, and (by Lemma~\ref{lem-bdry-porous}) $\bdry H$ is $L_1$-porous on scales
below $e^{-\eps d(w,H)}$.
Lemma~\ref{lem-parabolic-points} shows that $\cV$ is $L_4$-separated, for some $L_4$.

We set $L$ to be the maximum of $L_1$, $L_2$, $L_3$ and $L_4$.
We apply Corollary~\ref{cor-avoiding-qarc} to build a $\lambda$-quasi-arc $\gamma$ in $\bdry (G,\cP_1)$ for $\lambda=\lambda(L,N)$,
which satisfies, for all $H \in \cH$, 
\begin{equation}\label{eq-gam-sep} \rho(\gamma, \bdry H) \geq \frac{1}{\lambda}e^{-\eps d(w,H)}. \end{equation}

	In the Poincar\'e disc model for $\HH^2$, denote the standard half-space by 
	$Q = \{(x,y) : x^2+y^2 < 1, x \geq 0 \}$, and fix a basepoint $(0,0)$.
	We endow the semi-circle $\bdry Q$ with the angle metric $\rho_Q$, which makes $\bdry Q$
	quasi-symmetric (in fact similar) to the interval $[0,1]$.
	For some $C$, $\rho_Q$ is a visual metric on $\bdry Q$ with
	basepoint $(0,0)$ and parameters $C$ and $\eps=1$ \cite[III.H.3.19]{BH-99-Metric-spaces}.
	
	Therefore by \cite[Theorem 4.9]{TV-80-qs} there is a quasisymmetric map
	$f: \bdry Q \ra [0,1] \ra \gamma \subset \bdry (G,\cP)$.
	In fact, as $\bdry Q$ is connected, $f$ is a ``power quasisymmetry'' by \cite[Corollary 3.12]{TV-80-qs};
	see \cite[Section 6]{BS-00-gro-hyp-embed} for this definition.
	
	Both $Q$ and $X(G, \cP_1)$ are visual (see subsection \ref{ssec-visual}).
	Thus there is an extension of $f$ to a 
	quasi-isometric embedding of $Q$ in $X(G,\cP_1)$, with boundary $\gamma$ \cite[Theorems 7.4, 8.2]{BS-00-gro-hyp-embed}.
	
	Finally, as we have \eqref{eq-gam-sep}, Proposition \ref{prop-avoid-horoballs} 
	gives us a transversal, quasi-isometric embedding of $\HH^2$ in $X(G, \cP)$.
\end{proof}

\section{Application to 3-manifolds}\label{sec-three-manifolds}
In this final section, we consider which $3$-manifold groups contain a quasi-isometrically embedded copy
of $\HH^2$. Recall that an irreducible $3$-manifold is a graph manifold if its JSJ decomposition contains Seifert fibered components only. A non-geometric graph manifold is one with non-trivial JSJ decomposition.

\begin{lemma}\label{lem-graph-mfld-plane}
 Let $M$ be a non-geometric closed graph manifold. Then $\pi_1(M)$ contains a quasi-isometrically embedded copy of $\HH^2$.
\end{lemma}

\begin{proof}
 All fundamental groups of closed non-geometric graph manifolds are quasi-isometric 
 \cite[Theorem 2.1]{BN-08-qi-graph-manifold}, so we can choose $M$. 
 Consider a splitting of the closed genus 2 surface $S$ into an annulus $A$ and a twice-punctured torus $S'$, as in Figure~\ref{fig-surfacepic} below.
\begin{figure}[h]
\centering
\includegraphics[scale=0.5]{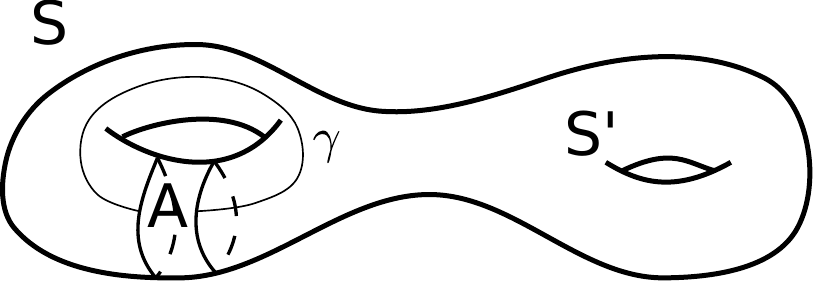}
\caption{The surfaces $S,S',A$ and the path $\gamma$.}\label{fig-surfacepic}
\end{figure}

The manifold $S' \times S^1$ has two boundary components homeomorphic to $S^1 \times S^1$.
Let $M$ be obtained from two copies $M_1,M_2$ of $S'\times S^1$ by gluing the corresponding boundary components together
in a way that interchanges the two $S^1$ factors.
\par
We now wish to find an embedding $\iota$ of $S$ into $M$ so that $M$ retracts onto the image of $\iota$. If we have such an embedding, then first of all $\pi_1(S)$ injects in $\pi_1(M)$, so that we get a map $f:\HH^2\to \tilM$. Also, $\pi_1(S)$ is undistorted in $\pi_1(M)$ and therefore $f$ is a quasi-isometric embedding.
\par
The specific embedding $\iota:S\to M$ which we describe is obtained from two embeddings $\iota_1:A\to M_1$ and $\iota_2:S'\to M_2$.
\par
Let $\gamma:[0,1]\to S'$ be the path connecting the boundary components of $S'$ depicted in Figure~\ref{fig-surfacepic}. As $A$ can be identified with $[0,1]\times S^1$ we can define $\iota_1:A\to M_1$ by $(t,\theta)\mapsto (\gamma(t),\theta)$.  We can assume, up to changing the gluings, that there exists $p$ such that $\iota_1(A)\cap M_2\subseteq S'\times\{p\}$. We can then define $\iota_2$ by $x\mapsto (x,p)$.
\par
We now have an embedding $\iota:S\to M$ so that $\iota|_A=\iota_1,\iota|_{S'}=\iota_2$. We now only need to show that $\iota(S)$ is a retract of $M$. Define $g_2:M_2\to S'\times\{p\}$ simply as $(x,\theta)\mapsto (x,p)$. It is easy to see that there exists a retraction $g':S'\to\gamma$ such that each boundary component of $S'$ is mapped to an endpoint of $\gamma$. Let $g_1:M_1\to \gamma\times S^1$ be $(x,\theta)\mapsto (g'(x),\theta)$. There clearly exists a retraction $g:M\to \iota(S)$ which coincides with $g_i$ on $M_i$.
\end{proof}

\begin{theorem}\label{thm-threemanifolds}
 Let $M$ be a connected orientable closed $3$-manifold. Then $\pi_1(M)$ contains a quasi-isometrically embedded copy of $\HH^2$ if and only if $M$ does not split as the connected sum
of manifolds each with geometry $S^3, \R^3, S^2\times \R$ or $\mathrm{Nil}$.
\end{theorem}

\begin{proof}
We will use the geometrisation theorem \cite{Pe1,Pe2,KL-notes,MT-poincare,CZ-geom}.
It is easily seen that $\pi_1(M)$ contains a quasi-isometrically embedded copy of $\HH^2$ if and only if the fundamental group of one of its prime summands does.
So, we can assume that $M$ is prime. Suppose first that $M$ is geometric. We list below the possible geometries, each followed by yes/no according to whether or not it contains a quasi-isometrically embedded copy of $\HH^2$ in that case and the reason for the answer.
\begin{itemize}
 \item $S^3$, no, it is compact.
 \item $\R^3$, no, it has polynomial growth.
 \item $\HH^3$, yes, obvious.
 \item $S^2\times \R$, no, it has linear growth.
 \item $\HH^2\times \R$, yes, obvious.
 \item $\widetilde {SL_2\R}$, yes, it is quasi-isometric to $\HH^2\times\R$ (see, for example, \cite{Rief-01-hyp-plane-x-R}).
 \item $\mathrm{Nil}$, no, it has polynomial growth.
 \item $\mathrm{Sol}$, yes, it contains isometrically embedded copies of $\HH^2$.
\end{itemize}

If $M$ is not geometric, then it has a non-trivial JSJ splitting, i.e. there is a canonical family of tori and Klein bottles that decomposes $M$ into components each of which is either Seifert fibered or hyperbolic (meaning that it admits a finite volume hyperbolic metric). We will consider two cases.

\begin{itemize}
 \item There are no hyperbolic components. By definition, $M$ is a graph manifold. In this case we can apply Lemma~\ref{lem-graph-mfld-plane} to find the quasi-isometrically embedded $\HH^2$.
 \item There is at least one hyperbolic component, $N$. As $\pi_1(N)$ is one-ended and hyperbolic relative to copies 
 of $\Z^2$, by Theorem \ref{thm-main-hyp-rel-hyp} 
 (or by \cite{Ma-Zh-08-fuch-hyp-knot,Ma-Zh-09-fuch-hyp-link}, upon applying Dehn filling to the manifold)
 it contains a quasi-isometrically embedded copy of $\HH^2$. 
 This is also quasi-isometrically embedded in $\pi_1(M)$ since $\pi_1(N)$ is undistorted in $\pi_1(M)$, because
 there exists a metric on $M$ such that $\widetilde{N}$ is convex in $\tilM$ (see \cite{Leeb-nonpos}). \qedhere
\end{itemize}
\end{proof}

%
\bibliographystyle{alpha}
\bibliography{biblio}

\end{document}